\documentclass[11pt]{article}

\usepackage[letterpaper,margin=1.1in]{geometry}
\usepackage{amsmath,amssymb,amsthm}
\usepackage{enumitem}
\usepackage[hidelinks]{hyperref}
\usepackage{microtype}

\newtheorem{theorem}{Theorem}
\newtheorem{proposition}{Proposition}
\newtheorem{lemma}{Lemma}
\newtheorem{corollary}{Corollary}
\newtheorem{definition}{Definition}
\newtheorem{remark}{Remark}

\DeclareMathOperator{\rank}{rank}
\DeclareMathOperator{\dist}{dist}
\DeclareMathOperator{\tr}{tr}
\DeclareMathOperator{\diag}{diag}
\DeclareMathOperator{\Var}{Var}
\DeclareMathOperator{\Cov}{Cov}

\newcommand{\R}{\mathbb R}
\newcommand{\Pbb}{\mathbb P}
\newcommand{\Ebb}{\mathbb E}
\newcommand{\1}{\mathbf 1}
\newcommand{\norm}[1]{\left\lVert #1\right\rVert}
\newcommand{\smin}{\sigma_{\min}}
\newcommand{\eps}{\varepsilon}

\title{Gaussian Critical-Threshold Instability in Real Phase Retrieval}
\author{Christian E. H{\"a}ggblom}
\date{}
\hypersetup{
  pdftitle={Gaussian Critical-Threshold Instability in Real Phase Retrieval},
  pdfauthor={Christian E. H{\"a}ggblom}
}

\begin{document}

\maketitle

\begin{abstract}
We prove an exponential instability theorem for real Gaussian phase retrieval
at the critical injectivity threshold.  If
\(A\in\mathbb R^{(2M-1)\times M}\) has iid standard Gaussian entries, then, for
every sequence \(w_M\to\infty\) satisfying \(w_M=e^{o(M)}\),
\[
        \Pbb\left\{
        \frac{1}{
        w_M\sqrt M\binom{2M-1}{M}}
        \le
        \omega(A)
        \le
        \frac{w_M}{
        \sqrt M\binom{2M-1}{M}}
        \right\}
        \longrightarrow1.
\]
Here \(\omega(A)\) is the Balan--Wang stability parameter.  The upper estimate
is the main result of the paper; the matching lower estimate follows
elementarily from a union bound and the square Gaussian hard-edge bound.
Consequently,
\[
        -\frac1M\log\omega(A)
        \longrightarrow
        \log4
\]
in probability.

The proof begins with a critical-threshold reduction showing that, for
full-spark matrices,
\[
        \omega(A)
        =
        L_A
        =
        \min_{|T|=M}\sigma_{\min}(A_T),
\]
where \(L_A\) is the optimal lower Lipschitz constant of the real phaseless map
\(x\mapsto |Ax|\).  The main probabilistic step studies the lower extreme of
the least singular values over all \(M\)-row minors.  A noncentral overlap
estimate gives the correct positive-probability scale.  The high-probability
upgrade is obtained from a central-overlap linearization: after conditioning
on the common rows of two minors, the problem reduces to a weighted
inverse-tail asymptotic for an independent square Gaussian block.  Gaussian
hard-edge estimates and inverse-Wishart concentration then yield asymptotic
independence in the central-overlap regime.
\end{abstract}

\section{Introduction}

Phase retrieval asks for the recovery of a vector \(x\in\mathbb R^M\) from
phaseless linear measurements
\[
        |Ax| = \bigl(|\langle A_1,x\rangle|,\ldots,
        |\langle A_N,x\rangle|\bigr),
\]
where \(A_1,\ldots,A_N\in\mathbb R^M\) are the measurement vectors. Since
\(x\) and \(-x\) generate the same data, injectivity is understood on the
quotient space \(\mathbb R^M/\{\pm1\}\).

In the real case, the sharp generic injectivity threshold is
\[
        N=2M-1.
\]
This threshold is characterized by the complement property: a real frame gives
injective phase retrieval if and only if, for every partition of the
measurement vectors, at least one side spans \(\mathbb R^M\). For background on
this criterion and its role in stability, see
\cite{BandeiraCahillMixonNelson2014}. At the critical threshold \(N=2M-1\),
generic frames are injective, but injectivity alone gives no quantitative
control on the inverse map.

The relevant stability question is therefore whether the phaseless map
\[
        \Phi_A:\mathbb R^M/\{\pm1\}\to\mathbb R^N,
        \qquad
        \Phi_A(x)=|Ax|,
\]
has a lower Lipschitz constant that remains quantitatively meaningful at the
critical threshold. Balan and Wang introduced a subset stability parameter
\[
        \omega(A)
        :=
        \min_{S\subset[N]:\operatorname{rank}(A_S)<M}
        \sigma_M(A_{S^c}),
\]
which gives a quantitative version of the complement property
\cite{BalanWang2015}. In the real setting, this quantity is directly tied to
the optimal lower Lipschitz constant of the phaseless map.

The specific critical-threshold Gaussian stability problem studied here was
brought to the author's attention through Afonso Bandeira's Randomstrasse101
entry ``Injectivity and Stability of Phase Retrieval.''  A stable version
appears as Entry~10, including Open Problem~21, in
\emph{Randomstrasse101: Open Problems of 2025}
\cite{BandeiraDmitrievLuccaNizicNikolacRodder2026}.  It formulates the Gaussian
version of the problem and asks for the typical behavior of \(\omega(A)\) when
\[
        A\in\mathbb R^{(2M-1)\times M}
\]
has iid standard Gaussian entries. The present paper identifies its natural scale up to arbitrary diverging
subexponential factors; the relation to independent work on the exponential
rate is discussed below.

Our first observation is a deterministic reduction at the critical threshold.
If \(A\in\mathbb R^{(2M-1)\times M}\) is full spark, then
\[
        \omega(A)
        =
        L_A
        =
        \min_{|T|=M}\sigma_{\min}(A_T),
\]
where \(L_A\) is the optimal lower Lipschitz constant of the real phaseless map
and \(A_T\) denotes the \(M\times M\) row submatrix indexed by \(T\). Thus the
critical-threshold stability problem reduces exactly to the lower extreme of
the least singular values of all square \(M\)-row minors.

The main result identifies the typical exponential scale of this lower extreme.  Let
\[
        K_M=\binom{2M-1}{M}.
\]
If \(A\in\mathbb R^{(2M-1)\times M}\) has iid standard Gaussian entries, then
for every sequence \(w_M\to\infty\) satisfying \(w_M=e^{o(M)}\),
\[
        \Pbb\left\{
        \frac{1}{w_M\sqrt M\,K_M}
        \le
        \omega(A)
        \le
        \frac{w_M}{\sqrt M\,K_M}
        \right\}
        \longrightarrow1.
\]
The upper bound is the substantive probabilistic result of the paper, while
the matching lower bound follows from a union bound and the square Gaussian
hard-edge estimate.  Since
\[
        \sqrt M\,K_M=4^{M+o(M)},
\]
it follows that
\[
        -\frac1M\log\omega(A)\longrightarrow\log4
\]
in probability.  Thus the critical exponential base for Gaussian frames is
\(1/4\), although the result does not assert concentration at an exact
subexponential prefactor.

The proof proceeds by a second-moment analysis over the family of all
\(M\times M\) minors. For a fixed \(M\)-subset \(T\), let
\[
        E_T(t)=
        \left\{\sigma_{\min}(A_T)\le \frac{t}{\sqrt M}\right\},
        \qquad
        Z_t=\sum_{|T|=M}\mathbf 1_{E_T(t)}.
\]
The first moment suggests the scale \(t\asymp K_M^{-1}\). To control the second
moment, pairs of minors are grouped according to their overlap. A non-centered
overlap estimate gives a sharp positive-probability theorem at the scale
\((\sqrt M K_M)^{-1}\).

The high-probability upgrade requires a sharper analysis of central overlaps.
After conditioning on the common rows of two minors and using rotational
invariance, the problem reduces to the block matrix
\[
        \begin{pmatrix}
        S&0\\
        G_1&G_2
        \end{pmatrix},
\]
where \(G_2\) is an independent square Gaussian block. The key local computation
is a weighted inverse-tail asymptotic for \(\|G_2^{-1}K\|\), with
\(K=(-G_1S^{-1},I)\). This computation is derived in the main text and yields a
conditional linearization of the bad-completion probability. Standard
inverse-Wishart concentration then shows that the resulting coefficient
self-averages in the central-overlap regime, giving asymptotic independence on
average and hence the high-probability theorem.

After completing this manuscript, the author became aware of independent work by
Yitzchak Shmalo~\cite{Shmalo2026} on extreme least singular values of Gaussian
row submatrices. In particular, Shmalo's Theorem~1.1 and Corollary~1.2 yield the exponential rate
\(1/4\) in the real critical phase-retrieval regime, as a specialization of a
more general fixed-aspect-ratio result. There is some overlap in the deterministic
critical-threshold reduction. Beyond this reduction, the probabilistic arguments
have different emphases: Shmalo's approach is directed toward identifying the
exponential rate in a more general fixed-aspect-ratio setting, whereas the present
argument focuses on the real critical threshold and develops the finer
high-probability localization at the natural scale
\[
    \frac{1}{\sqrt{M}\binom{2M-1}{M}},
\]
up to arbitrary diverging subexponential factors.

The rest of the paper is organized as follows. Section~\ref{sec:setup}
introduces the stability quantities and Section~\ref{sec:critical-threshold} proves the critical-threshold
reduction. Sections~\ref{sec:counting}--\ref{sec:positive} establish the
positive-probability theorem via the non-centered overlap profile. Section
\ref{sec:high-probability} proves the high-probability theorem using the
central-overlap linearization. The appendices record the standard Gaussian
hard-edge, rectangular hard-edge, inverse-Wishart, and hypergeometric estimates
used in the proof.

\section{Setup}
\label{sec:setup}

Let
\[
        A\in\R^{N\times M},
        \qquad
        A=\begin{pmatrix}A_1\\ \vdots\\ A_N\end{pmatrix},
\]
where \(N\ge M\ge1\) and \(A_i\in\R^M\) denotes the \(i\)-th row.  For \(S\subset[N]\), let \(A_S\)
denote the row submatrix indexed by \(S\).  Singular values are ordered decreasingly;
for an \(n\times M\) matrix \(B\), \(\sigma_M(B)^2=\lambda_{\min}(B^\top B)\), with
\(\sigma_M(B)=0\) if \(\rank(B)<M\).

\begin{definition}[Full spark]
A matrix \(A\in\R^{N\times M}\) is called full spark if every \(M\)-row submatrix is
nonsingular.  Equivalently,
\[
        \rank(A_S)=\min\{|S|,M\}
        \qquad\text{for every }S\subset[N].
\]
\end{definition}

The Balan--Wang subset stability quantity is
\[
        \omega(A)
        :=
        \min_{S\subset[N]:\ \rank(A_S)<M}
        \sigma_M(A_{S^c}).
\]
It is a quantitative form of the complement property for real phase retrieval.

Let
\[
        \Phi_A(x)=|Ax|,
        \qquad x\in\R^M,
\]
and equip \(\R^M/\{\pm1\}\) with the quotient metric
\[
        \dist(x,y)=\min_{\varepsilon\in\{\pm1\}}\norm{x-\varepsilon y}_2.
\]
Define the optimal Lipschitz constants
\[
        L_A
        :=
        \inf_{\dist(x,y)\ne0}
        \frac{\norm{|Ax|-|Ay|}_2}{\dist(x,y)},
        \qquad
        U_A
        :=
        \sup_{\dist(x,y)\ne0}
        \frac{\norm{|Ax|-|Ay|}_2}{\dist(x,y)}.
\]
The associated condition number is
\[
        \beta_A:=\frac{U_A}{L_A},
\]
with the convention \(\beta_A=+\infty\) when \(L_A=0\).

We use the following known formula of Balan and Wang; see \cite{BalanWang2015} and also \cite{XiaXuXu2024} for the condition-number formulation used here.

\begin{theorem}[Balan--Wang lower Lipschitz formula]
For real \(A\in\R^{N\times M}\),
\[
        U_A=\norm{A}_2
\]
and
\[
        L_A=\Delta_A,
\]
where
\[
        \Delta_A
        :=
        \min_{I\subset[N]}
        \sqrt{
        \lambda_{\min}(A_I^\top A_I)
        +
        \lambda_{\min}(A_{I^c}^\top A_{I^c})
        }.
\]
Consequently,
\[
        \beta_A=\frac{\norm{A}_2}{\Delta_A}.
\]
\end{theorem}

\section{Critical-threshold reduction}
\label{sec:critical-threshold}

We now specialize to the critical real threshold
\[
        N=2M-1.
\]

\begin{lemma}[Monotonicity under adding rows]
If \(R\subset T\subset[N]\), then
\[
        A_T^\top A_T-A_R^\top A_R=A_{T\setminus R}^\top A_{T\setminus R}\succeq0.
\]
Therefore
\[
        \sigma_M(A_T)\ge \sigma_M(A_R).
\]
\end{lemma}

\begin{proof}
The displayed identity is immediate from the row decomposition.  Since the difference is
positive semidefinite, Weyl monotonicity gives
\(\lambda_{\min}(A_T^\top A_T)\ge\lambda_{\min}(A_R^\top A_R)\).  Taking square roots
gives the result.
\end{proof}

\begin{theorem}[Critical-threshold reduction]
Let \(A\in\R^{(2M-1)\times M}\) be full spark.  Then
\[
        \boxed{
        \omega(A)=L_A=\Delta_A
        =
        \min_{\substack{T\subset[2M-1]\\ |T|=M}}
        \smin(A_T).
        }
\]
Consequently,
\[
        \beta_A
        =
        \frac{\norm{A}_2}{\omega(A)}.
\]
\end{theorem}

\begin{proof}
First consider \(\omega(A)\).  Since \(A\) is full spark,
\[
        \rank(A_S)<M
        \quad\Longleftrightarrow\quad
        |S|\le M-1.
\]
Therefore
\[
        \omega(A)
        =
        \min_{|S|\le M-1}\sigma_M(A_{S^c})
        =
        \min_{|T|\ge M}\sigma_M(A_T).
\]
By the monotonicity lemma, the latter minimum is attained among sets of cardinality
exactly \(M\). Hence
\[
        \omega(A)=\min_{|T|=M}\smin(A_T).
\]

Next, by the Balan--Wang formula,
\[
        L_A=\Delta_A
        =
        \min_{I\subset[2M-1]}
        \sqrt{
        \lambda_{\min}(A_I^\top A_I)
        +
        \lambda_{\min}(A_{I^c}^\top A_{I^c})
        }.
\]
For any partition \(I,I^c\) of \([2M-1]\), one side has at most \(M-1\) elements, and
therefore contributes zero to the displayed square root.  Thus
\[
        \Delta_A
        =
        \min_{|T|\ge M}\sigma_M(A_T)
        =
        \min_{|T|=M}\smin(A_T),
\]
again by monotonicity.  This proves all equalities.  Since \(U_A=\norm{A}_2\), the
identity for \(\beta_A\) follows.
\end{proof}

For Gaussian matrices, full spark holds almost surely.  Hence the critical-threshold
stability problem reduces to the lower extreme of the least singular values of all square
\(M\times M\) row submatrices.

\section{Gaussian counting setup}
\label{sec:counting}

Throughout the rest of the paper, let
\[
        A\in\R^{(2M-1)\times M}
\]
have iid \(N(0,1)\) entries.  Set
\[
        K_M:=\binom{2M-1}{M}.
\]
For \(T\subset[2M-1]\), \(|T|=M\), define
\[
        E_T(t)
        :=
        \left\{
        \smin(A_T)\le \frac{t}{\sqrt M}
        \right\},
        \qquad
        0<t<1,
\]
and
\[
        Z_t:=\sum_{|T|=M}\1_{E_T(t)}.
\]
Then
\[
        Z_t>0
        \quad\Longrightarrow\quad
        \omega(A)\le \frac{t}{\sqrt M}.
\]

For a fixed \(T\), \(A_T\) is an \(M\times M\) standard Gaussian matrix.  We will use the standard square least-singular-value estimate, recalled in Appendix~\ref{app:square-hard-edge},
\begin{equation}
        c_0 t
        \le
        \Pbb(E_T(t))
        \le
        C_0 t,
        \qquad 0<t<t_0,
        \label{eq:square-tail}
\end{equation}
with absolute constants \(c_0,C_0,t_0>0\).  In particular,
\[
        \Ebb Z_t=K_M\Pbb(E_T(t))\asymp K_M t.
\]
The scale suggested by the first moment is therefore
\[
        t\asymp K_M^{-1}.
\]

\section{Overlap decomposition of the second moment}

Fix \(T\subset[2M-1]\) with \(|T|=M\).  If \(U\) is another \(M\)-subset, write
\[
        s=|T\setminus U|=|U\setminus T|.
\]
Then \(s=0,1,\dots,M-1\).  For fixed \(T\), the number of \(U\)'s with this value of
\(s\) is
\[
        a_s:=\binom Ms\binom{M-1}{s}.
\]
For \(s\ge1\), define
\[
        q_{M,s}(t)
        :=
        \Pbb(E_T(t)\cap E_U(t)),
\]
where \(|T\setminus U|=s\).  By exchangeability, this depends only on \(M\) and \(s\).
Then
\begin{equation}
        \Ebb Z_t^2
        =
        K_M p_M(t)
        +
        K_M\sum_{s=1}^{M-1}a_s q_{M,s}(t),
        \label{eq:second-moment-profile}
\end{equation}
where
\[
        p_M(t):=\Pbb(E_T(t)).
\]
Thus the problem is to prove an aggregate estimate of the form
\begin{equation}
        \sum_{s=1}^{M-1}a_s q_{M,s}(t)
        \lesssim
        K_M t^2
        \label{eq:aggregate-target}
\end{equation}
for \(t\) of order \(K_M^{-1}\).

\section{Gaussian overlap estimate}

We now prove a quantitative pair-overlap estimate.  The proof isolates the only
non-elementary random matrix input: a hard-edge estimate for the smallest nonzero
singular value of a rectangular Gaussian matrix.

\subsection{Hard-edge input}

Let \(C\in\R^{(M-s)\times M}\) be a standard Gaussian matrix, and let
\[
        \mu_s(C):=\sigma_{M-s}(C)
\]
denote its smallest nonzero singular value.

\begin{lemma}[Rectangular hard-edge tail]
There are absolute constants \(C,c>0\) such that, for every \(1\le s\le M-1\) and
\(0<u\le c\),
\begin{equation}
        \Pbb\left\{
        \mu_s(C)
        \le
        u\frac{s}{\sqrt M}
        \right\}
        \le
        (Cu)^{s+1}.
        \label{eq:rect-tail}
\end{equation}
\end{lemma}

\begin{remark}
Estimate \eqref{eq:rect-tail} is proved directly for Gaussian matrices in
Appendix~\ref{app:rectangular-hard-edge} by integrating the real Wishart joint
density.  Its scale and hard-edge codimension exponent \(s+1\) agree with the
general rectangular smallest-singular-value estimates of Rudelson and Vershynin
\cite{RudelsonVershynin2009}.  The latter citation is included for comparison;
the exact pure-power estimate used here is supplied by the Gaussian derivation
in the appendix.
\end{remark}

The following inverse-moment consequence is the form used below.

\begin{lemma}[Truncated inverse moment]
Let \(C\in\R^{(M-s)\times M}\) be standard Gaussian and let
\(\mu=\mu_s(C)\).  There is an absolute constant \(C_1\) such that, for
\(2\le s\le M-1\) and all \(0<a\le c s/\sqrt M\),
\begin{equation}
        \Ebb\min\left\{1,\frac{a^2}{\mu^2}\right\}
        \le
        C_1 a^2\frac{M}{s^2}.
        \label{eq:inverse-moment}
\end{equation}
For \(s=1\), one has the logarithmic bound
\begin{equation}
        \Ebb\min\left\{1,\frac{a^2}{\mu^2}\right\}
        \le
        C_1 a^2 M\log\left(\frac{e}{a\sqrt M}\right)
        \label{eq:inverse-moment-s1}
\end{equation}
whenever \(0<a\le c/\sqrt M\).
\end{lemma}

\begin{proof}
We prove the case \(s\ge2\); the case \(s=1\) is identical except that the integral
\(\int_x^1 u^{-1}\,du\) produces a logarithm.
Set
\[
        X:=\frac{\sqrt M}{s}\mu.
\]
By \eqref{eq:rect-tail}, \(\Pbb\{X\le u\}\le (Cu)^{s+1}\) for \(0<u\le c\); decrease \(c\) so that \(Cc\le1/2\). Put
\(b:=a\sqrt M/s\).  Then
\[
        \min\left\{1,\frac{a^2}{\mu^2}\right\}
        =
        \min\left\{1,\frac{b^2}{X^2}\right\}.
\]
Using integration by parts for the decreasing function \(u\mapsto\min\{1,b^2/u^2\}\),
one obtains
\[
        \Ebb\min\left\{1,\frac{b^2}{X^2}\right\}
        =
        2b^2\int_b^\infty u^{-3}\Pbb\{X\le u\}\,du.
\]
The part \(u\le c\) is bounded by
\[
        2b^2 C^{s+1}\int_b^c u^{s-2}\,du
        \le \frac{2b^2 C^2(Cc)^{s-1}}{s-1}
        \lesssim b^2,
\]
uniformly for \(s\ge2\).  The part \(u>c\) is bounded by \(Cb^2\).  Hence the expectation is
\(O(b^2)=O(a^2M/s^2)\), proving \eqref{eq:inverse-moment}.
\end{proof}

\subsection{Block reduction for two overlapping minors}

Fix \(T,U\subset[2M-1]\) with \(|T|=|U|=M\) and
\(|T\setminus U|=s\).  Let
\[
        R=T\cap U,
        \qquad
        B=T\setminus U,
        \qquad
        D=U\setminus T.
\]
Then
\[
        |R|=M-s,
        \qquad
        |B|=|D|=s.
\]
Condition on the shared block
\[
        C:=A_R\in\R^{(M-s)\times M}.
\]
Given \(C\), the blocks \(A_B\) and \(A_D\) are independent standard Gaussian matrices.
Therefore
\begin{equation}
        q_{M,s}(t)
        =
        \Ebb_C\,q_C(t)^2,
        \label{eq:qCsquare}
\end{equation}
where
\[
        q_C(t)
        :=
        \Pbb_G\left\{
        \smin\begin{pmatrix}C\\G\end{pmatrix}
        \le
        \frac{t}{\sqrt M}
        \right\},
\]
and \(G\in\R^{s\times M}\) is standard Gaussian.

By rotational invariance, for almost every \(C\) there exists \(Q\in O(M)\) such that
\[
        CQ=\begin{pmatrix}S&0\end{pmatrix},
\]
where \(S\in\R^{(M-s)\times(M-s)}\) is invertible and
\(\smin(S)=\mu_s(C)\).  Since \(GQ\) is again standard Gaussian, write
\[
        GQ=\begin{pmatrix}G_1&G_2\end{pmatrix},
\]
where
\[
        G_1\in\R^{s\times(M-s)},
        \qquad
        G_2\in\R^{s\times s}
\]
are independent Gaussian blocks.  Thus
\[
        \smin\begin{pmatrix}C\\G\end{pmatrix}
        =
        \smin
        \begin{pmatrix}
        S&0\\G_1&G_2
        \end{pmatrix}.
\]

\begin{lemma}[Block inverse implication]
Let
\[
        L=
        \begin{pmatrix}
        S&0\\G_1&G_2
        \end{pmatrix},
\]
where \(S\) and \(G_2\) are square and invertible.  Set
\[
        \mu=\smin(S),
        \qquad
        \nu=\smin(G_2),
        \qquad
        H=\norm{G_1}_2.
\]
If \(\smin(L)\le\eps\) and \(\mu>2\eps\), then
\begin{equation}
        \nu
        \le
        2\eps\left(1+\frac{H}{\mu}\right).
        \label{eq:block-implication}
\end{equation}
\end{lemma}

\begin{proof}
The inverse is
\[
        L^{-1}
        =
        \begin{pmatrix}
        S^{-1}&0\\
        -G_2^{-1}G_1S^{-1}&G_2^{-1}
        \end{pmatrix}.
\]
Viewing \(L^{-1}\) as the vertical concatenation of its two block rows and using
the triangle inequality gives the exact estimate
\begin{align*}
        \norm{L^{-1}}
        &\le
        \norm{(S^{-1},0)}
        +
        \norm{G_2^{-1}(-G_1S^{-1},I_s)}
        \\
        &\le
        \mu^{-1}
        +
        \nu^{-1}\left(1+\frac{H}{\mu}\right).
\end{align*}
If \(\smin(L)\le\eps\), then \(\norm{L^{-1}}\ge\eps^{-1}\).  Since \(\mu>2\eps\),
\(\mu^{-1}<1/(2\eps)\), and therefore
\[
        \frac{1}{2\eps}
        \le
        \nu^{-1}\left(1+\frac{H}{\mu}\right).
\]
Rearranging gives \eqref{eq:block-implication}.
\end{proof}

\begin{lemma}[Conditional small-ball estimate]
There is an absolute constant \(C\) such that, for every full-row-rank shared block \(C\),
with \(\mu=\mu_s(C)\),
\begin{equation}
        q_C(t)
        \le
        C\min\left\{
        1,
        t\sqrt{\frac{s}{M}}
        \left(1+\frac{\sqrt M}{\mu}\right)
        \right\}.
        \label{eq:conditional-qC}
\end{equation}
\end{lemma}

\begin{proof}
Set \(\eps=t/\sqrt M\).  By the block inverse implication,
\[
        \{\smin(L)\le\eps\}
        \subset
        \{\mu\le2\eps\}
        \cup
        \left\{
        \smin(G_2)
        \le
        C\eps\left(1+\frac{\norm{G_1}}{\mu}\right)
        \right\}.
\]
The square Gaussian least-singular-value upper tail gives
\[
        \Pbb\{\smin(G_2)\le u\}\le C\sqrt s\,u
\]
for \(G_2\in\R^{s\times s}\).  Conditioning on \(G_1\) and using
\(\Ebb\norm{G_1}\le C\sqrt M\), we obtain
\[
        q_C(t)
        \le
        \1_{\{\mu\le2t/\sqrt M\}}
        +
        C\frac{t\sqrt s}{\sqrt M}
        \left(1+\frac{\sqrt M}{\mu}\right).
\]
The right side is bounded by a constant multiple of the truncated expression in
\eqref{eq:conditional-qC}; when \(\mu\le2t/\sqrt M\), the second term is already larger
than a positive absolute constant after truncation.  This proves the claim.
\end{proof}

\begin{proposition}[Overlap profile bound]
\label{prop:overlap-profile-bound}
There is an absolute constant \(C\) such that, for \(2\le s\le M-1\) and
\(0<t\le c\sqrt{s/M}\),
\begin{equation}
        q_{M,s}(t)
        \le
        C t^2\frac{M}{s}.
        \label{eq:qMs-main}
\end{equation}
For \(s=1\) and \(0<t\le c/\sqrt M\),
\begin{equation}
        q_{M,1}(t)
        \le
        C t^2 M\log\left(\frac{e}{t\sqrt M}\right).
        \label{eq:qMs-s1}
\end{equation}
\end{proposition}

\begin{proof}
From \eqref{eq:qCsquare} and \eqref{eq:conditional-qC},
\[
        q_{M,s}(t)
        \le
        C\,\Ebb_C
        \min\left\{
        1,
        t^2\frac{s}{M}
        \left(1+\frac{M}{\mu^2}\right)
        \right\}.
\]
The term without \(\mu\) is bounded by \(Ct^2s/M\), which is at most
\(Ct^2M/s\).  The singular term is bounded by
\[
        C\,\Ebb_C
        \min\left\{
        1,
        \frac{t^2s}{\mu^2}
        \right\}.
\]
Apply the truncated inverse-moment lemma with \(a=t\sqrt s\).  The hypothesis
\(t\le c\sqrt{s/M}\) is exactly the required condition
\(a\le cs/\sqrt M\).  For \(s\ge2\), this gives
\[
        \Ebb_C
        \min\left\{1,\frac{t^2s}{\mu^2}\right\}
        \le
        C(t^2s)\frac{M}{s^2}
        =
        C t^2\frac{M}{s},
\]
which proves \eqref{eq:qMs-main}.  The case \(s=1\) follows from
\eqref{eq:inverse-moment-s1} and gives the logarithmic estimate \eqref{eq:qMs-s1}.
\end{proof}

\section{Combinatorial summation}

We now show that the overlap-profile losses are absorbed by the number of pairs at each
overlap level.

\begin{lemma}[Overlap summation]
There is an absolute constant \(C\) such that
\begin{equation}
        \sum_{s=1}^{M-1}\binom Ms\binom{M-1}s\frac{M}{s}
        \le
        C\binom{2M-1}{M}
        =C K_M.
        \label{eq:comb-sum}
\end{equation}
Moreover, for \(t=\lambda/K_M\) with fixed \(\lambda>0\), the exceptional logarithmic
\(s=1\) contribution satisfies, for all sufficiently large \(M\),
\begin{equation}
        \binom M1\binom{M-1}1\,t^2 M
        \log\left(\frac{e}{t\sqrt M}\right)
        \le
        C K_M t^2.
        \label{eq:s1-negligible}
\end{equation}
\end{lemma}

\begin{proof}
Let \(U\) be uniformly distributed over all \(M\)-subsets of \([2M-1]\), with \(T\) fixed.
Then
\[
        \Pbb\{|T\setminus U|=s\}
        =
        \frac{\binom Ms\binom{M-1}s}{K_M}.
\]
Thus the left side of \eqref{eq:comb-sum} equals
\[
        K_M\,\Ebb\left[\frac{M}{S}\1_{S\ge1}\right],
        \qquad S:=|T\setminus U|.
\]
On \(\{S\ge M/4\}\), \(M/S\le4\).  On \(\{1\le S<M/4\}\), the hypergeometric lower-tail
probability is exponentially small in \(M\), while \(M/S\le M\).  Hence
\[
        \Ebb\left[\frac{M}{S}\1_{S\ge1}\right]\le C.
\]
This proves \eqref{eq:comb-sum}.  For \eqref{eq:s1-negligible}, observe that the left side
is
\[
        O\left(M^3t^2\log K_M\right),
\]
whereas \(K_Mt^2\) has an exponentially larger coefficient in \(M\) because
\(K_M\asymp 4^M/\sqrt M\).  Therefore the inequality holds for all sufficiently large
\(M\).
\end{proof}

\section{Near-sharp Gaussian theorem}
\label{sec:positive}

\begin{theorem}[Positive-probability Gaussian instability]
There exist absolute constants \(C<\infty\), \(p_0>0\), and \(M_0\) such that for all
\(M\ge M_0\), if
\[
        A\in\R^{(2M-1)\times M}
\]
has iid standard Gaussian entries, then
\begin{equation}
        \Pbb\left\{
        \omega(A)
        \le
        \frac{C}{\sqrt M\binom{2M-1}{M}}
        \right\}
        \ge
        p_0.
        \label{eq:gaussian-theorem-omega}
\end{equation}
Consequently, after increasing \(C\) and decreasing \(p_0\) if necessary,
\begin{equation}
        \Pbb\left\{
        \frac{\omega(A)}{\max_{k\in[2M-1]}\norm{A_k}_2}
        \le
        \frac{C}{M\binom{2M-1}{M}}
        \right\}
        \ge
        p_0.
        \label{eq:gaussian-theorem-normalized}
\end{equation}
\end{theorem}

\begin{proof}
Let
\[
        t=\frac{\lambda}{K_M},
\]
where \(\lambda>0\) is a fixed constant to be chosen large.  By the square
least-singular-value lower tail \eqref{eq:square-tail},
\[
        \Ebb Z_t
        =K_Mp_M(t)
        \ge
        c_0K_Mt
        =c_0\lambda.
\]
On the other hand, using \eqref{eq:second-moment-profile}, the overlap profile bound, and
the combinatorial summation lemma,
\[
        \Ebb Z_t^2
        \le
        C\left(K_Mt+K_M^2t^2\right)
        \le
        C(\lambda+\lambda^2).
\]
Therefore, by Paley--Zygmund,
\[
        \Pbb\{Z_t>0\}
        \ge
        \frac{(\Ebb Z_t)^2}{\Ebb Z_t^2}
        \ge
        \frac{c\lambda^2}{C(\lambda+\lambda^2)}.
\]
Choosing \(\lambda\) fixed and sufficiently large gives a positive lower bound
\(p_0>0\), independent of \(M\).  On \(\{Z_t>0\}\),
\[
        \omega(A)
        \le
        \frac{t}{\sqrt M}
        =
        \frac{\lambda}{K_M\sqrt M},
\]
which proves \eqref{eq:gaussian-theorem-omega}.

For the normalized statement, note that
\[
        \max_{k\in[2M-1]}\norm{A_k}_2\ge c\sqrt M
\]
with probability tending to one.  Intersecting this event with \eqref{eq:gaussian-theorem-omega}
and adjusting constants gives \eqref{eq:gaussian-theorem-normalized}.
\end{proof}

\section{High-probability theorem via central-overlap linearization}
\label{sec:high-probability}

The positive-probability theorem is obtained by bounding the non-centered second moment.
We now record a high-probability upgrade.  The proof uses the same overlap decomposition,
but strengthens the central-overlap estimate to asymptotic independence.

Throughout this section let
\[
        t_M:=\frac{w_M}{K_M},
        \qquad
        K_M=\binom{2M-1}{M},
\]
where
\[
        w_M\longrightarrow\infty,
        \qquad
        w_M=e^{o(M)}.
\]
The subexponential restriction is harmless for the intended applications; it allows, for
example, powers of \(M\) and powers of \(\log M\).

For two \(M\)-subsets \(T,U\), recall that
\[
        s=|T\setminus U|=|U\setminus T|,
        \qquad
        |T\cap U|=M-s.
\]
Fix once and for all a number \(\eta\in(0,1/2)\), and call an overlap central if
\[
        s\in\mathcal C_{\eta,M}:=\bigl\{s:\eta M\le s\le (1-\eta)M\bigr\}.
\]
The noncentral values of \(s\) have exponentially small combinatorial mass and will be
handled by the non-centered overlap bound from Proposition~\ref{prop:overlap-profile-bound}.

\subsection{A specialized Gaussian inverse-tail input}

We first isolate the only random matrix input beyond the standard least-singular-value
bounds already used above.  This is a specialized inverse-tail statement for a square
Gaussian matrix multiplied by a row-conditioned matrix.  It is a direct consequence of the
Gaussian singular-value decomposition, the independence of singular vectors and singular
values, and the hard-edge expansion for the least singular value.

Let \(X_s\in\R^{s\times s}\) be standard Gaussian.  Write
\[
        \rho_s:=\sigma_{\min}(X_s).
\]
Let \(\gamma_s\) denote the density-at-zero coefficient for \(\rho_s\), namely
\begin{equation}
        \Pbb\{\rho_s\le r\}
        =
        \gamma_s r+o(r),
        \qquad r\downarrow0.
        \label{eq:gamma-s-def}
\end{equation}
For real Gaussian matrices,
\begin{equation}
        \frac{\gamma_s}{\sqrt{s}}
        \longrightarrow
        \gamma_\ast=1.
        \label{eq:gamma-limit}
\end{equation}

For \(K\in\R^{s\times n}\) and \(r>0\), define
\[
        \Psi_s(K;r)
        :=
        \Pbb_{X_s}\{\|X_s^{-1}K\|\ge r\},
        \qquad
        a_s(K)
        :=
        \Ebb_u\|u^\top K\|_2,
\]
where \(u\) is Haar-uniform on \(S^{s-1}\).

\begin{lemma}[Uniform specialized inverse-tail asymptotic]
\label{lem:specialized-inverse-tail}
Let \(X_s\in\R^{s\times s}\) be standard Gaussian, and let
\(K_s\in\R^{s\times n_s}\) be independent of \(X_s\). Assume that
\[
        K_sK_s^\top\succeq I_s
\]
and that, for every fixed \(0<p<\infty\), there is \(s_0(p)\) such that
\begin{equation}
        \Ebb\|K_s\|^p\le C_p,\qquad s\ge s_0(p)
        \label{eq:K-moment-bound}
\end{equation}
with constants independent of \(s\). All conclusions below are as \(s\to\infty\).
Let \(R_s\to\infty\) satisfy
\[
        R_s\ge e^{cs}
\]
for some fixed \(c>0\). Then
\begin{equation}
        \Ebb_{K_s}
        \sup_{r\in[R_s/2,\,2R_s]}
        \left|
        \Psi_s(K_s;r)
        -
        \frac{\gamma_s}{r}a_s(K_s)
        \right|^2
        =
        o\left(\frac{s}{R_s^2}\right).
        \label{eq:uniform-specialized-inverse-tail}
\end{equation}
\end{lemma}

\begin{proof}
All constants below may depend on the fixed moment order used in the
truncation, but not on \(s\), \(R_s\), or
\(r\in[R_s/2,2R_s]\). We write \(C_0\) for a fixed absolute
polynomial exponent in the hard-edge remainder estimates, while \(C\)
denotes constants that may change from line to line.

Write the singular-value decomposition
\[
        X_s=U\Sigma V^\top,
        \qquad
        \Sigma=\diag(\rho_1,\ldots,\rho_s),
        \qquad
        0<\rho_1\le\rho_2\le\cdots\le\rho_s.
\]
Then
\[
        \|X_s^{-1}K_s\|
        =
        \|\Sigma^{-1}U^\top K_s\|.
\]
For a real Gaussian matrix, \(U\) is Haar orthogonal and is independent
of the singular values. Let \(u_1\) denote its first column,
corresponding to \(\rho_1\).

Fix \(\delta\in(0,1/10)\), and first condition on a deterministic value
\(K_s=K\) satisfying
\[
        \|K\|\le R_s^\delta.
\]
On the event
\[
        \rho_2\ge R_s^{-1+3\delta},
\]
the contribution from the singular directions \(2,\ldots,s\) is
bounded by
\[
        \rho_2^{-1}\|K\|
        \le R_s^{1-2\delta}.
\]
Since \(r\ge R_s/2\), this contribution is \(o(r)\), uniformly for
\(r\in[R_s/2,2R_s]\). In particular, for all sufficiently large \(s\),
\begin{align*}
        \left\{
        \rho_1
        \le
        \frac{\|u_1^\top K\|_2}
        {r(1+R_s^{-\delta})}
        \right\}
        &\subset
        \{\|X_s^{-1}K\|\ge r\}
\\
        &\subset
        \left\{
        \rho_1
        \le
        \frac{\|u_1^\top K\|_2}
        {r(1-R_s^{-\delta})}
        \right\}
\end{align*}
on \(\{\rho_2\ge R_s^{-1+3\delta}\}\).

Using the Haar independence of \(u_1\), it follows that
\begin{equation}
        \Psi_s(K;r)
        =
        \Ebb_u
        \Pbb\left\{
        \rho_1
        \le
        \frac{\|u^\top K\|_2}{r}
        \right\}
        +
        \mathcal E_s(K,r),
        \label{eq:inverse-tail-first-direction}
\end{equation}
where the estimate for \(\mathcal E_s(K,r)\) is uniform over
\(r\in[R_s/2,2R_s]\).

Indeed, the error in \eqref{eq:inverse-tail-first-direction} is bounded
by the sum of the level-repulsion probability
\[
        \Pbb\{\rho_2\le R_s^{-1+3\delta}\}
\]
and the error produced by changing the threshold for \(\rho_1\) by the
relative factor \(1+O(R_s^{-\delta})\). By the hard-edge estimates in
Appendix~\ref{app:square-hard-edge}, for
\(0\le x\le C R_s^{-1+\delta}\),
\[
        \Pbb\{\rho_1\le x\}
        =
        \gamma_s x+O(s^{C_0}x^2),
        \qquad
        \Pbb\{\rho_2\le x\}
        =
        O(s^{C_0}x^3).
\]
Since
\[
        \frac{\|u^\top K\|_2}{r}
        \le
        2R_s^{-1+\delta},
\]
these estimates apply uniformly on the truncation event. The same hard-edge
bounds also apply at \(x=R_s^{-1+3\delta}\), since this cutoff tends to zero.

Consequently,
\begin{align}
        \sup_{r\in[R_s/2,\,2R_s]}
        |\mathcal E_s(K,r)|
        \le{}&
        C s^{C_0}R_s^{-3+9\delta}
        +
        C\gamma_sR_s^{-1-\delta}a_s(K)
\nonumber\\
        &+
        C s^{C_0}R_s^{-2+2\delta}
        \Ebb_u\|u^\top K\|_2^2.
        \label{eq:uniform-first-direction-error}
\end{align}
Moreover, expanding the principal term in
\eqref{eq:inverse-tail-first-direction} at the hard edge gives
\begin{align}
        \Ebb_u
        \Pbb\left\{
        \rho_1
        \le
        \frac{\|u^\top K\|_2}{r}
        \right\}
        &=
        \frac{\gamma_s}{r}a_s(K)
        +
        O\left(
        \frac{s^{C_0}}{r^2}
        \Ebb_u\|u^\top K\|_2^2
        \right).
        \label{eq:hard-edge-principal-expansion}
\end{align}

Define
\[
        D_s(K;r)
        :=
        \Psi_s(K;r)
        -
        \frac{\gamma_s}{r}a_s(K).
\]
Since \(r\ge R_s/2\), the bounds
\eqref{eq:uniform-first-direction-error} and
\eqref{eq:hard-edge-principal-expansion} imply, on
\(\{\|K\|\le R_s^\delta\}\),
\begin{align}
        \sup_{r\in[R_s/2,\,2R_s]}
        |D_s(K;r)|
        \le{}&
        C s^{C_0}R_s^{-3+9\delta}
        +
        C\gamma_sR_s^{-1-\delta}a_s(K)
\nonumber\\
        &+
        C s^{C_0}R_s^{-2+2\delta}
        \Ebb_u\|u^\top K\|_2^2.
        \label{eq:combined-specialized-error}
\end{align}
Here the \(r^{-2}\) remainder in
\eqref{eq:hard-edge-principal-expansion} has been absorbed into the
last term.

Let
\[
        \Omega_s:=\{\|K_s\|\le R_s^\delta\}.
\]
Using
\[
        (x+y+z)^2\le 3(x^2+y^2+z^2),
        \qquad
        a_s(K)\le\|K\|,
        \qquad
        \Ebb_u\|u^\top K\|_2^2\le\|K\|^2,
\]
and then applying \eqref{eq:K-moment-bound} with \(p=2\) and \(p=4\),
we obtain
\begin{align*}
        &\Ebb\left[
        \mathbf 1_{\Omega_s}
        \sup_{r\in[R_s/2,\,2R_s]}
        |D_s(K_s;r)|^2
        \right]
\\
        &\qquad\le
        C s^{2C_0}R_s^{-6+18\delta}
        +
        C\gamma_s^2R_s^{-2-2\delta}
        \Ebb a_s(K_s)^2
\\
        &\hspace{4.2cm}
        +
        C s^{2C_0}R_s^{-4+4\delta}
        \Ebb\left(
        \Ebb_u\|u^\top K_s\|_2^2
        \right)^2
\\
        &\qquad\le
        C s^{2C_0}R_s^{-6+18\delta}
        +
        C\gamma_s^2R_s^{-2-2\delta}
        \Ebb\|K_s\|^2
        +
        C s^{2C_0}R_s^{-4+4\delta}
        \Ebb\|K_s\|^4.
\end{align*}
Since \(\gamma_s^2\asymp s\), \(R_s\ge e^{cs}\), and
\(\delta<1/10\), all three terms are \(o(s/R_s^2)\). Hence
\begin{equation}
        \Ebb\left[
        \mathbf 1_{\Omega_s}
        \sup_{r\in[R_s/2,\,2R_s]}
        |D_s(K_s;r)|^2
        \right]
        =
        o\left(\frac{s}{R_s^2}\right).
        \label{eq:truncated-specialized-L2}
\end{equation}

It remains to control the truncation complement. Uniformly for
\(r\in[R_s/2,2R_s]\),
\[
        0\le\Psi_s(K_s;r)\le1,
        \qquad
        \frac{\gamma_s}{r}a_s(K_s)
        \le
        \frac{2\gamma_s}{R_s}\|K_s\|.
\]
Therefore,
\begin{align*}
        &\Ebb\left[
        \mathbf 1_{\Omega_s^c}
        \sup_{r\in[R_s/2,\,2R_s]}
        |D_s(K_s;r)|^2
        \right]
\\
        &\qquad\le
        2\Pbb(\Omega_s^c)
        +
        \frac{8\gamma_s^2}{R_s^2}
        \Ebb\left[
        \|K_s\|^2\mathbf 1_{\Omega_s^c}
        \right].
\end{align*}
Choose a fixed \(p>\max\{2,2/\delta\}\). By
\eqref{eq:K-moment-bound},
\[
        \Pbb(\Omega_s^c)
        \le
        R_s^{-p\delta}\Ebb\|K_s\|^p
        \le
        C_pR_s^{-p\delta}
        =
        o\left(\frac{s}{R_s^2}\right).
\]
Likewise,
\[
        \Ebb\left[
        \|K_s\|^2\mathbf 1_{\Omega_s^c}
        \right]
        \le
        R_s^{-\delta(p-2)}
        \Ebb\|K_s\|^p
        \le
        C_pR_s^{-\delta(p-2)}.
\]
Since \(\gamma_s^2\asymp s\), it follows that
\[
        \frac{\gamma_s^2}{R_s^2}
        \Ebb\left[
        \|K_s\|^2\mathbf 1_{\Omega_s^c}
        \right]
        =
        o\left(\frac{s}{R_s^2}\right).
\]
Combining this estimate with
\eqref{eq:truncated-specialized-L2} proves
\eqref{eq:uniform-specialized-inverse-tail}.
\end{proof}

The uniformity in Lemma~\ref{lem:specialized-inverse-tail} allows the deterministic
threshold \(R_s\) to be perturbed by a random amount that is only polynomially large.

\begin{lemma}[Quantitative stability under an additive threshold shift]
\label{lem:inverse-tail-threshold-shift}
Assume the hypotheses of Lemma~\ref{lem:specialized-inverse-tail}.  Let \(D_s\ge0\) be
jointly distributed with \(K_s\), independent of \(X_s\), and suppose that there is a
constant \(A<\infty\) such that, for every fixed \(0<p<\infty\) and all sufficiently large \(s\),
\begin{equation}
        \Ebb D_s^p\le C_p s^{Ap}.
        \label{eq:threshold-shift-moments}
\end{equation}
Let \(\theta_s\) be any random variable jointly distributed with \((K_s,D_s)\), independent
of \(X_s\), and satisfying \(|\theta_s|\le1\).  Set
\[
        \widetilde R_s
        :=
        R_s+\theta_sD_s.
\]
Then
\begin{equation}
        \Ebb_{K_s,D_s,\theta_s}
        \left|
        \Psi_s(K_s;\widetilde R_s)
        -
        \Psi_s(K_s;R_s)
        \right|^2
        =
        o\left(\frac{s}{R_s^2}\right).
        \label{eq:quantitative-threshold-shift}
\end{equation}
Here, on the negligible event \(\widetilde R_s\le0\), the expression
\(\Psi_s(K_s;\widetilde R_s)\) is interpreted as \(1\).
\end{lemma}

\begin{proof}
Let
\[
        \Omega_s:=\{D_s\le R_s/2\}.
\]
On \(\Omega_s\),
\[
        \widetilde R_s\in[R_s/2,\,3R_s/2].
\]
Define
\[
        \varepsilon_s(K;r)
        :=
        \Psi_s(K;r)-\frac{\gamma_s}{r}a_s(K).
\]
Then, on \(\Omega_s\),
\begin{align*}
        \left|
        \Psi_s(K_s;\widetilde R_s)
        -
        \Psi_s(K_s;R_s)
        \right|
        \le{}&
        |\varepsilon_s(K_s;\widetilde R_s)|
        +
        |\varepsilon_s(K_s;R_s)|
\\
        &+
        \gamma_s a_s(K_s)
        \left|
        \frac{1}{\widetilde R_s}
        -
        \frac{1}{R_s}
        \right|.
\end{align*}
The first two terms contribute \(o(s/R_s^2)\) in squared mean by
Lemma~\ref{lem:specialized-inverse-tail}.  For the last term, on \(\Omega_s\),
\[
        \left|
        \frac{1}{\widetilde R_s}
        -
        \frac{1}{R_s}
        \right|
        \le
        \frac{2D_s}{R_s^2}.
\]
Since \(a_s(K_s)\le\|K_s\|\), Hölder's inequality,
\eqref{eq:K-moment-bound}, and \eqref{eq:threshold-shift-moments} give
\[
        \Ebb
        \left[
        \gamma_s^2a_s(K_s)^2
        \left|
        \frac{1}{\widetilde R_s}
        -
        \frac{1}{R_s}
        \right|^2
        1_{\Omega_s}
        \right]
        \le
        \frac{C\gamma_s^2}{R_s^4}
        \Ebb\!\left[\|K_s\|^2D_s^2\right]
        \le
        \frac{C s^{2A+1}}{R_s^4}.
\]
Because \(R_s\ge e^{cs}\), this is \(o(s/R_s^2)\).

On \(\Omega_s^c\), the difference of the two probabilities is bounded by \(1\).  For every
fixed \(p\),
\[
        \Pbb(\Omega_s^c)
        =
        \Pbb\{D_s>R_s/2\}
        \le
        \frac{2^p\Ebb D_s^p}{R_s^p}
        \le
        C_p\frac{s^{Ap}}{R_s^p}.
\]
Choosing \(p>2\) shows that this is \(o(s/R_s^2)\).  Combining the estimates on
\(\Omega_s\) and \(\Omega_s^c\) proves
\eqref{eq:quantitative-threshold-shift}.
\end{proof}

\begin{corollary}[Random-threshold inverse-tail asymptotic]
\label{cor:random-threshold-inverse-tail}
Under the hypotheses of Lemma~\ref{lem:inverse-tail-threshold-shift},
\begin{equation}
        \Ebb_{K_s,D_s,\theta_s}
        \left|
        \Pbb_{X_s}
        \{\|X_s^{-1}K_s\|\ge R_s+\theta_sD_s\}
        -
        \frac{\gamma_s}{R_s}
        \Ebb_u\|u^\top K_s\|_2
        \right|^2
        =
        o\left(\frac{s}{R_s^2}\right).
        \label{eq:random-threshold-inverse-tail}
\end{equation}
\end{corollary}

\begin{proof}
By the triangle inequality in \(L^2\), combine
Lemma~\ref{lem:inverse-tail-threshold-shift} with
Lemma~\ref{lem:specialized-inverse-tail} evaluated at \(r=R_s\).
\end{proof}

\begin{remark}
The preceding results make explicit the threshold replacement used below.  In the
row-conditioned application,
\[
        K=(-G_1S^{-1},I_s),
        \qquad
        D_s=\|S^{-1}\|.
\]
The matrix \(K\) and the shift \(D_s\) may be dependent, but both are independent of the
square Gaussian block \(G_2\).  In the central-overlap regime,
Corollaries~\ref{cor:central-inverse-operator-moments}
and~\ref{cor:row-conditioned-factor-moments} verify the moment assumptions
for \(D_s\) and \(K\), respectively.  Since the inverse-tail threshold is exponential
in \(s\), while \(D_s\) has bounded polynomial moments, the relative displacement
\(D_s/R_s\) is exponentially small in every fixed \(L^p\).

The local computation itself uses only the Gaussian singular-value decomposition, the
independence of singular vectors and singular values, and the two quantitative hard-edge
estimates recorded in Appendix~\ref{app:square-hard-edge}.  The uniform formulation above
also permits the sandwich
\[
        \|G_2^{-1}K\|
        \le
        \|L^{-1}\|
        \le
        \|G_2^{-1}K\|+\|S^{-1}\|
\]
to be used without an implicit random-threshold substitution.
\end{remark}

\subsection{Conditional linearization of the bad-completion probability}

Fix a central value \(s\in\mathcal C_{\eta,M}\), and set
\[
        m:=M-s.
\]
Let
\[
        C\in\R^{m\times M}
\]
be a standard Gaussian shared block, and let
\[
        G\in\R^{s\times M}
\]
be an independent Gaussian completion.  Define
\[
        q_C(t):=
        \Pbb_G\left\{
        \smin\begin{pmatrix}C\\G\end{pmatrix}
        \le \frac{t}{\sqrt M}
        \right\}.
\]

\begin{lemma}[Conditional linearization]
\label{lem:conditional-linearization}
Uniformly for \(s\in\mathcal C_{\eta,M}\),
\begin{equation}
        q_C(t_M)
        =
        t_M\Theta_{M,s}(C)+r_{M,s}(C),
        \label{eq:conditional-linearization}
\end{equation}
where
\begin{equation}
        \Ebb_C|r_{M,s}(C)|^2=o(t_M^2),
        \label{eq:conditional-remainder-L2}
\end{equation}
and
\begin{equation}
        \Theta_{M,s}(C)
        :=
        \frac{\gamma_s}{\sqrt M}
        \Ebb_z
        \sqrt{1+z^\top(CC^\top)^{-1}z},
        \qquad
        z\sim N(0,I_m).
        \label{eq:theta-def}
\end{equation}
\end{lemma}

\begin{proof}
By rotational invariance, choose \(Q\in O(M)\), measurably as a function of \(C\), such
that
\[
        CQ=(S,0),
\]
where \(S\in\R^{m\times m}\) is invertible almost surely.  In particular,
\[
        CC^\top=SS^\top.
\]
Conditionally on \(C\), the matrix \(GQ\) is again standard Gaussian and is independent of
\(C\).  Write
\[
        GQ=(G_1,G_2),
\]
where
\[
        G_1\in\R^{s\times m},
        \qquad
        G_2\in\R^{s\times s}
\]
are independent standard Gaussian blocks.  Thus
\[
        \smin\begin{pmatrix}C\\G\end{pmatrix}
        \stackrel{d}{=}
        \smin L,
        \qquad
        L:=
        \begin{pmatrix}
        S&0\\
        G_1&G_2
        \end{pmatrix}.
\]

The matrices \(S\) and \(G_2\) are invertible almost surely, and
\[
        L^{-1}
        =
        \begin{pmatrix}
        S^{-1}&0\\
        -G_2^{-1}G_1S^{-1}&G_2^{-1}
        \end{pmatrix}.
\]
Set
\[
        K:=(-G_1S^{-1},I_s)\in\R^{s\times M},
        \qquad
        D:=\|S^{-1}\|.
\]
The lower block of \(L^{-1}\) is \(G_2^{-1}K\).  Consequently,
\begin{equation}
        \|G_2^{-1}K\|
        \le
        \|L^{-1}\|
        \le
        \|G_2^{-1}K\|+D.
        \label{eq:block-inverse-norm-equivalence}
\end{equation}

Let
\[
        R_M:=\frac{\sqrt M}{t_M}.
\]
Since
\[
        \left\{\smin(L)\le\frac{t_M}{\sqrt M}\right\}
        =
        \{\|L^{-1}\|\ge R_M\},
\]
we may write
\[
        q_C(t_M)
        =
        \Ebb_{G_1}
        \Pbb_{G_2}\{\|L^{-1}\|\ge R_M\mid C,G_1\}.
\]

For fixed \(C\) and \(G_1\), define
\[
        \Psi_{C,G_1}(r)
        :=
        \Pbb_{G_2}\{\|G_2^{-1}K\|\ge r\}.
\]
The norm comparison \eqref{eq:block-inverse-norm-equivalence} gives the pointwise
probability sandwich
\begin{equation}
        \Psi_{C,G_1}(R_M)
        \le
        \Pbb_{G_2}\{\|L^{-1}\|\ge R_M\mid C,G_1\}
        \le
        \Psi_{C,G_1}(R_M-D).
        \label{eq:block-inverse-probability-sandwich}
\end{equation}
On the negligible event \(R_M-D\le0\), the right-hand side is interpreted as \(1\).

We now verify the hypotheses of the inverse-tail results from the preceding subsection.
Since \(s\in\mathcal C_{\eta,M}\), one has \(s\asymp_\eta M\).  Moreover,
\[
        t_M=\frac{w_M}{K_M},
        \qquad
        w_M=e^{o(M)},
\]
so there is a constant \(c_\eta>0\) such that
\[
        R_M\ge e^{c_\eta s}
\]
for all sufficiently large \(M\), uniformly over
\(s\in\mathcal C_{\eta,M}\).

The matrix \(K\) is independent of \(G_2\) and satisfies
\[
        KK^\top\succeq I_s.
\]

Because \(S\) has the same singular values as \(C\),
Corollary~\ref{cor:central-inverse-operator-moments} gives, for every fixed
\(0<p<\infty\),
\begin{equation}
        \Ebb_C D^p
        =
        \Ebb_C\|S^{-1}\|^p
        =
        O_{\eta,p}(M^{-p/2}).
        \label{eq:D-central-moments}
\end{equation}
Moreover, Corollary~\ref{cor:row-conditioned-factor-moments} gives
\begin{equation}
        \Ebb_{C,G_1}\|K\|^p
        \le
        C_{\eta,p}.
        \label{eq:K-central-moments}
\end{equation}

All moment bounds here are needed only for fixed orders and sufficiently
large \(M\), uniformly over central \(s\). The estimates in the inverse-tail
proof depend only on these uniform constants and the exponential lower bound
for \(R_M\); their remainders are therefore uniform over central overlaps.
Thus the additive-threshold estimate
Lemma~\ref{lem:inverse-tail-threshold-shift}, applied jointly to the randomness
\((C,G_1)\) with \(X_s=G_2\), \(D_s=D\), and \(\theta_s=-1\), yields
\begin{equation}
        \Ebb_{C,G_1}
        \left|
        \Psi_{C,G_1}(R_M-D)
        -
        \Psi_{C,G_1}(R_M)
        \right|^2
        =
        o\left(\frac{s}{R_M^2}\right).
        \label{eq:block-shift-L2}
\end{equation}
Because \(s\asymp_\eta M\) and \(R_M=\sqrt M/t_M\),
\[
        \frac{s}{R_M^2}
        =
        \frac{s}{M}t_M^2
        \asymp_\eta t_M^2.
\]
Hence the right-hand side of \eqref{eq:block-shift-L2} is \(o(t_M^2)\).

Define
\[
        \widehat q_C
        :=
        \Ebb_{G_1}\Psi_{C,G_1}(R_M).
\]
By the sandwich \eqref{eq:block-inverse-probability-sandwich},
\[
        0
        \le
        q_C(t_M)-\widehat q_C
        \le
        \Ebb_{G_1}
        \left[
        \Psi_{C,G_1}(R_M-D)
        -
        \Psi_{C,G_1}(R_M)
        \right].
\]
Jensen's inequality and \eqref{eq:block-shift-L2} therefore imply
\begin{equation}
        \Ebb_C
        |q_C(t_M)-\widehat q_C|^2
        =
        o(t_M^2).
        \label{eq:qC-block-reduction-error}
\end{equation}

It remains to linearize \(\widehat q_C\).  Apply
Lemma~\ref{lem:specialized-inverse-tail} jointly to the random matrix
\[
        K=K(C,G_1),
\]
which is independent of \(G_2\).  The joint moment estimate
\eqref{eq:K-central-moments} is the relevant hypothesis; no uniform conditional
moment bound for every fixed \(C\) is needed.  We obtain
\begin{equation}
        \Psi_{C,G_1}(R_M)
        =
        \frac{\gamma_s}{R_M}
        \Ebb_u\|u^\top K\|_2
        +
        \xi_{M,s}(C,G_1),
        \label{eq:conditional-inverse-tail-expansion}
\end{equation}
where
\begin{equation}
        \Ebb_{C,G_1}
        |\xi_{M,s}(C,G_1)|^2
        =
        o\left(\frac{s}{R_M^2}\right)
        =
        o(t_M^2).
        \label{eq:conditional-inverse-tail-error}
\end{equation}

Averaging \eqref{eq:conditional-inverse-tail-expansion} over \(G_1\) gives
\[
        \widehat q_C
        =
        \frac{\gamma_s}{R_M}
        \Ebb_{G_1,u}\|u^\top K\|_2
        +
        \overline\xi_{M,s}(C),
\]
where
\[
        \overline\xi_{M,s}(C)
        :=
        \Ebb_{G_1}[\xi_{M,s}(C,G_1)\mid C].
\]
By Jensen's inequality and \eqref{eq:conditional-inverse-tail-error},
\begin{equation}
        \Ebb_C|\overline\xi_{M,s}(C)|^2
        =
        o(t_M^2).
        \label{eq:averaged-inverse-tail-error}
\end{equation}

Since \(\|u\|_2=1\),
\[
        \|u^\top K\|_2
        =
        \sqrt{1+\|u^\top G_1S^{-1}\|_2^2}.
\]
Also,
\[
        \frac{1}{R_M}
        =
        \frac{t_M}{\sqrt M}.
\]
It follows from \eqref{eq:qC-block-reduction-error} and
\eqref{eq:averaged-inverse-tail-error} that
\begin{equation}
        q_C(t_M)
        =
        \frac{\gamma_s t_M}{\sqrt M}
        \Ebb_{G_1,u}
        \sqrt{1+\|u^\top G_1S^{-1}\|_2^2}
        +
        r_{M,s}(C),
        \label{eq:qC-before-Gaussian-average}
\end{equation}
with
\[
        \Ebb_C|r_{M,s}(C)|^2=o(t_M^2).
\]

For fixed \(u\in S^{s-1}\), the row vector \(u^\top G_1\) is distributed as
\(z^\top\), where \(z\sim N(0,I_m)\).  Hence
\[
        \|u^\top G_1S^{-1}\|_2^2
        \stackrel{d}{=}
        z^\top S^{-1}S^{-\top}z.
\]
The matrices
\[
        S^{-1}S^{-\top}
        \qquad\text{and}\qquad
        (SS^\top)^{-1}=S^{-\top}S^{-1}
\]
have the same eigenvalues.  Since \(z\) is rotationally invariant and
\(SS^\top=CC^\top\),
\[
        \Ebb_{G_1,u}
        \sqrt{1+\|u^\top G_1S^{-1}\|_2^2}
        =
        \Ebb_z
        \sqrt{1+z^\top(CC^\top)^{-1}z}.
\]
Substituting this identity into \eqref{eq:qC-before-Gaussian-average} proves
\eqref{eq:conditional-linearization}--\eqref{eq:theta-def}.
\end{proof}

\subsection{Concentration of the coefficient}

\begin{lemma}[Inverse-Wishart trace concentration]
\label{lem:inverse-wishart-concentration}
Let \(C\in\R^{m\times M}\) be standard Gaussian, where
\[
        m=M-s,
        \qquad
        \frac{s}{M}\in[\eta,1-\eta].
\]
Set \(W:=CC^\top\) and \(\beta:=s/M\).  Then
\begin{equation}
        \tr(W^{-1})=\frac{1-\beta}{\beta}+o_{L^2}(1),
        \label{eq:tr-W-inv-conc}
\end{equation}
and
\begin{equation}
        \tr(W^{-2})=o_{L^1}(1),
        \label{eq:tr-W-inv2}
\end{equation}
uniformly for \(s\in\mathcal C_{\eta,M}\).
\end{lemma}

\begin{proof}
The matrix \(W\) is Wishart with dimension \(m\) and \(M\) degrees of freedom.  The inverse-Wishart moment formulas recalled in Appendix~\ref{app:inverse-wishart}, which are standard consequences of \cite{Matsumoto2012,LetacMassam2004}, give
\[
        \Ebb W^{-1}=\frac{I_m}{M-m-1}=\frac{I_m}{s-1},
\]
and hence
\[
        \Ebb\tr(W^{-1})=\frac{m}{s-1}=\frac{1-\beta}{\beta}+o(1).
\]
The second-order inverse-Wishart moment formulas give
\[
        \Var(\tr(W^{-1}))=O_\eta(M^{-2})
\]
and
\[
        \Ebb\tr(W^{-2})=O_\eta(M^{-1}).
\]
These estimates imply \eqref{eq:tr-W-inv-conc} and \eqref{eq:tr-W-inv2}.
\end{proof}

\begin{lemma}[Coefficient concentration]
\label{lem:theta-concentration}
Uniformly for \(s\in\mathcal C_{\eta,M}\),
\begin{equation}
        \Theta_{M,s}(C)=\gamma_\ast+o_{L^2}(1).
        \label{eq:theta-conc}
\end{equation}
In particular,
\begin{equation}
        \Var_C(\Theta_{M,s}(C))\longrightarrow0.
        \label{eq:theta-var-zero}
\end{equation}
\end{lemma}

\begin{proof}
Let \(W=CC^\top\) and \(z\sim N(0,I_m)\), independent of \(W\).  Conditional on \(W\),
\[
        \Ebb_z z^\top W^{-1}z=\tr(W^{-1})
\]
and
\[
        \Var_z(z^\top W^{-1}z)=2\tr(W^{-2}).
\]
By Lemma~\ref{lem:inverse-wishart-concentration},
\[
        z^\top W^{-1}z
        =\frac{1-\beta}{\beta}+o_{L^2}(1).
\]
The map \(x\mapsto\sqrt{1+x}\) is Lipschitz on \([0,\infty)\), so
\[
        \sqrt{1+z^\top W^{-1}z}
        =
        \frac{1}{\sqrt\beta}+o_{L^2(W,z)}(1).
\]
Moreover, Jensen's inequality makes the passage to the conditional expectation
explicit:
\begin{align*}
        &\Ebb_W
        \left|
        \Ebb_z\sqrt{1+z^\top W^{-1}z}
        -
        \frac{1}{\sqrt\beta}
        \right|^2
        \\
        &\qquad\le
        \Ebb_{W,z}
        \left|
        \sqrt{1+z^\top W^{-1}z}
        -
        \frac{1}{\sqrt\beta}
        \right|^2
        =o(1).
\end{align*}
Consequently,
\[
        \Ebb_z\sqrt{1+z^\top W^{-1}z}
        =
        \frac{1}{\sqrt\beta}+o_{L^2(W)}(1).
\]
Therefore, by \eqref{eq:theta-def},
\[
        \Theta_{M,s}(C)
        =
        \frac{\gamma_s}{\sqrt M}
        \left(\frac{1}{\sqrt\beta}+o_{L^2}(1)\right).
\]
Since \(s=\beta M\) and \(\gamma_s/\sqrt s\to\gamma_\ast\),
\[
        \frac{\gamma_s}{\sqrt M}\frac1{\sqrt\beta}
        =
        \frac{\gamma_s}{\sqrt s}
        \longrightarrow \gamma_\ast.
\]
This proves \eqref{eq:theta-conc} and hence \eqref{eq:theta-var-zero}.
\end{proof}

\subsection{Central-overlap asymptotic independence}

\begin{proposition}[Central-overlap asymptotic independence]
\label{prop:central-asymptotic-independence}
Let \(t_M=w_M/K_M\), where \(w_M\to\infty\) and \(w_M=e^{o(M)}\).  Then, uniformly for
\(s\in\mathcal C_{\eta,M}\),
\begin{equation}
        q_{M,s}(t_M)=(1+o(1))p_M(t_M)^2.
        \label{eq:central-asymptotic-independence}
\end{equation}
\end{proposition}

\begin{proof}
For pairs with difference size \(s\), conditioning on the shared block \(C\) gives
\[
        q_{M,s}(t_M)=\Ebb_C q_C(t_M)^2,
        \qquad
        p_M(t_M)=\Ebb_C q_C(t_M).
\]
By Lemma~\ref{lem:conditional-linearization},
\[
        q_C(t_M)=t_M\Theta_{M,s}(C)+r_{M,s}(C),
        \qquad
        \Ebb_C|r_{M,s}(C)|^2=o(t_M^2).
\]
By Lemma~\ref{lem:theta-concentration},
\[
        \Theta_{M,s}(C)=\gamma_\ast+o_{L^2}(1).
\]
Hence
\[
        q_{M,s}(t_M)
        =t_M^2\Ebb_C\Theta_{M,s}(C)^2+o(t_M^2)
        =t_M^2\gamma_\ast^2+o(t_M^2),
\]
and similarly
\[
        p_M(t_M)
        =t_M\Ebb_C\Theta_{M,s}(C)+o(t_M)
        =t_M\gamma_\ast+o(t_M).
\]
Dividing the two asymptotics gives \eqref{eq:central-asymptotic-independence}.
\end{proof}

\subsection{High-probability instability theorem}

\begin{theorem}[High-probability Gaussian instability]
\label{thm:high-probability-gaussian-instability}
Let
\[
        A\in\R^{(2M-1)\times M}
\]
have iid standard Gaussian entries.  Let \(w_M\to\infty\) satisfy \(w_M=e^{o(M)}\).  Then
\begin{equation}
        \Pbb\left\{
        \omega(A)
        \le
        \frac{w_M}{\sqrt M\binom{2M-1}{M}}
        \right\}
        \longrightarrow 1.
        \label{eq:hp-omega}
\end{equation}
Consequently,
\begin{equation}
        \Pbb\left\{
        \frac{\omega(A)}{\max_{k\in[2M-1]}\|A_k\|_2}
        \le
        \frac{w_M}{M\binom{2M-1}{M}}
        \right\}
        \longrightarrow 1.
        \label{eq:hp-normalized}
\end{equation}
\end{theorem}

\begin{proof}
Let
\[
        Z_t=\sum_{|T|=M}\1_{E_T(t)},
        \qquad
        E_T(t)=\left\{\smin(A_T)\le\frac{t}{\sqrt M}\right\}.
\]
Set \(t=t_M=w_M/K_M\).  Since \(t_M\to0\), the triangular-array estimate
\eqref{eq:triangular-square-hard-edge} in
Corollary~\ref{cor:uniform-square-hard-edge-scaling} gives
\[
        p_M(t_M):=\Pbb(E_T(t_M))=(1+o(1))t_M.
\]
Therefore
\[
        \Ebb Z_t=K_Mp_M(t)\sim w_M\to\infty.
\]
The variance decomposes as
\begin{equation}
        \Var(Z_t)
        \le
        K_Mp_M(t)
        +K_M\sum_{s=1}^{M-1}a_s\bigl(q_{M,s}(t)-p_M(t)^2\bigr),
        \label{eq:variance-profile-hp}
\end{equation}
where
\[
        a_s=\binom Ms\binom{M-1}s.
\]
Divide by \((\Ebb Z_t)^2=K_M^2p_M(t)^2\).  The diagonal term gives
\[
        \frac{K_Mp_M(t)}{K_M^2p_M(t)^2}
        =
        \frac{1}{K_Mp_M(t)}
        \sim
        \frac{1}{w_M}
        \longrightarrow0.
\]

For central overlaps \(s\in\mathcal C_{\eta,M}\), Proposition~\ref{prop:central-asymptotic-independence} gives
\[
        q_{M,s}(t)-p_M(t)^2=o(p_M(t)^2)
\]
uniformly.  Therefore
\[
        \frac1{K_Mp_M(t)^2}
        \sum_{s\in\mathcal C_{\eta,M}}
        a_s\bigl(q_{M,s}(t)-p_M(t)^2\bigr)
        =o(1),
\]
because \(\sum_{s=1}^{M-1}a_s=K_M-1\).

For noncentral overlaps, use the non-centered profile bound from
Proposition~\ref{prop:overlap-profile-bound} and the endpoint estimate \eqref{eq:qMs-s1}.  Since
\(p_M(t)\asymp t\), this gives
\[
        q_{M,s}(t)\le C p_M(t)^2\frac{M}{s}
\]
for \(s\ge2\), with a logarithmic endpoint correction at \(s=1\).  If \(S\) is the
overlap-difference parameter for a uniformly random \(M\)-subset \(U\), with \(T\) fixed,
then
\[
        \Pbb\{S=s\}=\frac{a_s}{K_M}.
\]
This hypergeometric variable is concentrated around \(M/2\).  Hence
\[
        \frac1{K_M}
        \sum_{\substack{1\le s\le M-1\\s\notin\mathcal C_{\eta,M}}}
        a_s\frac{M}{s}
        \longrightarrow0,
\]
while the \(s=1\) logarithmic term is exponentially negligible compared with \(K_M\)
even after multiplication by \(\log(1/t_M)=O(M+\log w_M)\).  Thus the noncentral
contribution to \eqref{eq:variance-profile-hp}, normalized by \((\Ebb Z_t)^2\), is also
\(o(1)\).

Consequently,
\[
        \frac{\Var(Z_t)}{(\Ebb Z_t)^2}\longrightarrow0.
\]
By Chebyshev's inequality,
\[
        \Pbb\{Z_t=0\}
        \le
        \frac{\Var(Z_t)}{(\Ebb Z_t)^2}
        \longrightarrow0.
\]
Therefore \(Z_t>0\) with probability tending to one.  On this event, at least one
\(M\times M\) row minor satisfies
\[
        \smin(A_T)\le\frac{t_M}{\sqrt M}.
\]
By the critical-threshold reduction,
\[
        \omega(A)=\min_{|T|=M}\smin(A_T),
\]
and hence
\[
        \omega(A)\le\frac{t_M}{\sqrt M}
        =
        \frac{w_M}{\sqrt M K_M}
\]
with probability tending to one.  This proves \eqref{eq:hp-omega}.

Finally,
\[
        \max_{k\in[2M-1]}\|A_k\|_2\ge \tfrac12\sqrt M
\]
with probability tending to one, since even the first row has \(\chi^2_M\) norm squared.
Apply \eqref{eq:hp-omega} with \(w_M/2\) in place of \(w_M\), and intersect the two
events. This proves \eqref{eq:hp-normalized} exactly as stated.
\end{proof}

\begin{corollary}[Typical exponential rate]
\label{cor:typical-exponential-rate}
Let
\[
        K_M=\binom{2M-1}{M},
\]
and let \(A\in\mathbb R^{(2M-1)\times M}\) have iid standard Gaussian
entries. For every sequence \(w_M\to\infty\) satisfying
\(w_M=e^{o(M)}\),
\[
        \mathbb P\left\{
        \frac{1}{w_M\sqrt M\,K_M}
        \le \omega(A)
        \le
        \frac{w_M}{\sqrt M\,K_M}
        \right\}
        \longrightarrow 1.
\]
Consequently,
\[
        -\frac{1}{M}\log\omega(A)
        \longrightarrow \log 4
\]
in probability.
\end{corollary}

\begin{proof}
The upper bound is the preceding theorem. For the lower bound, use the
critical-threshold identity
\[
        \omega(A)=\min_{|T|=M}\sigma_{\min}(A_T)
\]
and the union bound:
\[
\begin{aligned}
        \mathbb P\left\{
        \omega(A)\le
        \frac{1}{w_M\sqrt M\,K_M}
        \right\}
        &\le
        K_M
        \mathbb P\left\{
        \sigma_{\min}(G_M)
        \le
        \frac{1}{w_M\sqrt M\,K_M}
        \right\}.
\end{aligned}
\]
The square Gaussian hard-edge estimate
\[
        \mathbb P\left\{
        \sigma_{\min}(G_M)\le \frac{x}{\sqrt M}
        \right\}\le Cx
\]
therefore gives
\[
        \mathbb P\left\{
        \omega(A)\le
        \frac{1}{w_M\sqrt M\,K_M}
        \right\}
        \le \frac{C}{w_M}\longrightarrow0.
\]
Finally,
\[
        \log(\sqrt M\,K_M)=M\log4+O(\log M),
\]
while \(\log w_M=o(M)\), which proves the logarithmic convergence.
\end{proof}

\newpage

\appendix
\section{Square Gaussian hard-edge facts}
\label{app:square-hard-edge}

This appendix derives the square Gaussian hard-edge estimates used in the paper.  Let
\[
        G_s\in\R^{s\times s}
\]
be standard Gaussian, and write
\[
        0<\rho_1\le\rho_2\le\cdots\le\rho_s
\]
for its ordered singular values.

We use the empty-dimensional convention \(\mathcal Z_{0,N}=1\).
For integers \(1\le m\le N\), define the singular-value partition function
\begin{equation}
        \mathcal Z_{m,N}
        :=
        \int_{(0,\infty)^m}
        e^{-\frac12\sum_{i=1}^m x_i^2}
        \prod_{i=1}^m x_i^{N-m}
        \prod_{1\le i<j\le m}|x_j^2-x_i^2|
        \,dx_1\cdots dx_m.
        \label{eq:singular-value-partition-function}
\end{equation}
The Laguerre--Selberg integral gives
\begin{equation}
        \mathcal Z_{m,N}
        =
        2^{mN/2-m}
        \prod_{j=1}^m
        \frac{
        \Gamma(1+j/2)
        \Gamma((N-m+j)/2)
        }{
        \Gamma(3/2)
        }.
        \label{eq:singular-value-partition-evaluation}
\end{equation}
Equivalently, the ordered singular values of an \(m\times N\) standard Gaussian matrix
have density
\begin{equation}
        \frac{m!}{\mathcal Z_{m,N}}
        e^{-\frac12\sum_{i=1}^m x_i^2}
        \prod_{i=1}^m x_i^{N-m}
        \prod_{1\le i<j\le m}(x_j^2-x_i^2)
        \label{eq:ordered-singular-value-density}
\end{equation}
on the chamber
\[
        0<x_1<\cdots<x_m.
\]
These formulas are the singular-value form of the real Laguerre orthogonal ensemble; see,
for example, \cite{Edelman1988,TaoVu2010}.

\begin{lemma}[Square Gaussian hard-edge package]
\label{lem:square-hard-edge-package}
The following facts hold for the real Gaussian ensemble.

\begin{enumerate}[label=(\alph*)]
    \item One may choose the singular-value decomposition
    \[
            G_s=U\Sigma V^\top
    \]
    so that \(U\) and \(V\) are Haar orthogonal and are independent of the singular
    values.  In particular, the left singular direction corresponding to \(\rho_1\) is
    Haar-uniform and independent of all singular values.

    \item The density-at-zero coefficient for \(\rho_1\) is
    \begin{equation}
            \gamma_s
            =
            \frac{s}{\sqrt2}
            \frac{\Gamma((s+1)/2)}{\Gamma((s+2)/2)}.
            \label{eq:gamma-s-explicit}
    \end{equation}
    There are absolute constants \(C,r_0>0\) such that
    \begin{equation}
            0
            \le
            \gamma_s r-\Pbb\{\rho_1\le r\}
            \le
            Csr^2,
            \qquad
            0\le r\le r_0.
            \label{eq:uniform-square-hard-edge-expansion}
    \end{equation}
    Moreover,
    \begin{equation}
            \frac{\gamma_s}{\sqrt{s}}
            =
            1+O(s^{-1}),
            \label{eq:gamma-s-asymptotic}
    \end{equation}
    and hence
    \[
            \frac{\gamma_s}{\sqrt{s}}\longrightarrow1.
    \]

    \item For \(s\ge2\) and every \(r\ge0\),
    \begin{equation}
            \Pbb\{\rho_2\le r\}
            \le
            \frac{s(s-1)}{12}r^4.
            \label{eq:second-singular-value-r4}
    \end{equation}
    In particular, after taking \(r_0\le1\),
    \begin{equation}
            \Pbb\{\rho_2\le r\}
            \le
            Cs^2r^3,
            \qquad
            0\le r\le r_0.
            \label{eq:second-singular-value-r3}
    \end{equation}
\end{enumerate}
\end{lemma}

\begin{proof}
The density of \(G_s\) is invariant under
\[
        G_s\longmapsto Q_1G_sQ_2,
        \qquad
        Q_1,Q_2\in O(s).
\]
Consequently, conditionally on the singular values, the left and right singular directions
are Haar-distributed.  After resolving the harmless sign ambiguity in the SVD by independent
random signs, \(U\) and \(V\) may be taken Haar and independent of the singular values.  This
proves part~(a).

We next derive the density-at-zero coefficient and its uniform remainder.  Let \(h_s(t)\)
denote the density of \(\rho_1\).  From
\eqref{eq:ordered-singular-value-density},
\begin{align}
        h_s(t)
        &=
        \frac{s!}{\mathcal Z_{s,s}}
        e^{-t^2/2}
        \int_{t<x_2<\cdots<x_s}
        e^{-\frac12\sum_{j=2}^s x_j^2}
        \prod_{j=2}^s(x_j^2-t^2)
        \prod_{2\le i<j\le s}(x_j^2-x_i^2)
        \,dx_2\cdots dx_s.
        \label{eq:smallest-singular-density}
\end{align}
At \(t=0\), the remaining integral is
\[
        \frac{1}{(s-1)!}\mathcal Z_{s-1,s+1}.
\]
Therefore,
\begin{equation}
        \gamma_s
        :=
        h_s(0)
        =
        s\frac{\mathcal Z_{s-1,s+1}}{\mathcal Z_{s,s}}.
        \label{eq:gamma-s-partition-ratio}
\end{equation}
Using \eqref{eq:singular-value-partition-evaluation} and canceling the common factors,
\[
        \frac{\mathcal Z_{s-1,s+1}}{\mathcal Z_{s,s}}
        =
        \frac{1}{\sqrt2}
        \frac{\Gamma((s+1)/2)}{\Gamma((s+2)/2)}.
\]
Substitution into \eqref{eq:gamma-s-partition-ratio} proves
\eqref{eq:gamma-s-explicit}.

For the quantitative expansion, let
\[
        H\in\R^{(s-1)\times(s+1)}
\]
be standard Gaussian, and write
\[
        0<\tau_1<\cdots<\tau_{s-1}
\]
for its singular values.  Factoring \(x_j^2\) from each term \(x_j^2-t^2\) in
\eqref{eq:smallest-singular-density} and using
\eqref{eq:gamma-s-partition-ratio}, we obtain the exact identity
\begin{equation}
        \frac{h_s(t)}{\gamma_s}
        =
        e^{-t^2/2}
        \Ebb_H
        \left[
        1_{\{\tau_1>t\}}
        \prod_{j=1}^{s-1}
        \left(1-\frac{t^2}{\tau_j^2}\right)
        \right].
        \label{eq:smallest-density-rectangular-representation}
\end{equation}
For \(s=1\), the same identity is interpreted as
\[
        h_1(t)=\gamma_1e^{-t^2/2},
\]
and the estimates below follow directly.  We therefore assume \(s\ge2\).

Every factor inside the expectation in
\eqref{eq:smallest-density-rectangular-representation} lies in \([0,1]\).  Hence
\begin{equation}
        0\le h_s(t)\le\gamma_s.
        \label{eq:smallest-density-monotonic-upper-bound}
\end{equation}
Furthermore, using
\[
        1-\prod_{j=1}^{s-1}(1-y_j)
        \le
        \sum_{j=1}^{s-1}y_j,
        \qquad 0\le y_j\le1,
\]
we obtain
\begin{align}
        0
        \le
        1-\frac{h_s(t)}{\gamma_s}
        \le{}&
        \frac{t^2}{2}
        +
        \Pbb\{\tau_1\le t\}
        +
        t^2\Ebb\sum_{j=1}^{s-1}\tau_j^{-2}.
        \label{eq:smallest-density-defect}
\end{align}

The matrix \(HH^\top\) is Wishart with dimension \(s-1\) and \(s+1\) degrees of freedom.
The standard inverse-Wishart mean formula gives
\begin{equation}
        \Ebb(HH^\top)^{-1}
        =
        I_{s-1},
        \qquad
        \Ebb\sum_{j=1}^{s-1}\tau_j^{-2}
        =
        s-1.
        \label{eq:codimension-two-inverse-trace}
\end{equation}
The rectangular Gaussian hard-edge estimate from
Lemma~\ref{lem:rectangular-gaussian-hard-edge}, applied with column dimension \(s+1\)
and rectangular deficit \(2\), gives
\begin{equation}
        \Pbb\{\tau_1\le t\}
        \le
        C s^{3/2}t^3,
        \qquad
        0\le t\le\frac{c}{\sqrt{s}}.
        \label{eq:codimension-two-smallest-tail}
\end{equation}
Combining \eqref{eq:smallest-density-defect}--\eqref{eq:codimension-two-smallest-tail},
and decreasing \(c\) if necessary, yields
\begin{equation}
        0
        \le
        \gamma_s-h_s(t)
        \le
        C\gamma_s s t^2,
        \qquad
        0\le t\le\frac{c}{\sqrt{s}}.
        \label{eq:smallest-density-local-defect}
\end{equation}

The explicit formula \eqref{eq:gamma-s-explicit} and standard gamma-ratio bounds imply
\[
        c\sqrt{s}\le\gamma_s\le C\sqrt{s}.
\]
Integrating \eqref{eq:smallest-density-local-defect} over \(0\le t\le r\) gives
\begin{equation}
        0
        \le
        \gamma_s r-\Pbb\{\rho_1\le r\}
        \le
        Cs^{3/2}r^3,
        \qquad
        0\le r\le\frac{c}{\sqrt{s}}.
        \label{eq:smallest-cdf-local-remainder}
\end{equation}
Since \(s^{3/2}r^3\le Csr^2\) in this range, this proves
\eqref{eq:uniform-square-hard-edge-expansion} for
\(r\le c/\sqrt{s}\).

For \(c/\sqrt{s}\le r\le r_0\), the upper bound
\eqref{eq:smallest-density-monotonic-upper-bound} gives
\[
        0
        \le
        \gamma_s r-\Pbb\{\rho_1\le r\}
        \le
        \gamma_s r
        \le
        C\sqrt{s}\,r
        \le
        Csr^2.
\]
Thus \eqref{eq:uniform-square-hard-edge-expansion} holds throughout
\(0\le r\le r_0\).

Finally, the standard gamma-ratio expansion
\[
        \frac{\Gamma(x+1/2)}{\Gamma(x+1)}
        =
        x^{-1/2}\left(1+O(x^{-1})\right),
        \qquad x\to\infty,
\]
applied with \(x=s/2\), gives
\[
        \gamma_s
        =
        \frac{s}{\sqrt2}
        \left(\frac{s}{2}\right)^{-1/2}
        \left(1+O(s^{-1})\right)
        =
        \sqrt{s}\left(1+O(s^{-1})\right).
\]
This proves \eqref{eq:gamma-s-asymptotic} and part~(b).

It remains to prove the two-singular-value estimate.  Assume \(s\ge2\).  In the ordered
density \eqref{eq:ordered-singular-value-density}, write
\[
        t:=\rho_1,
        \qquad
        u:=\rho_2.
\]
On the chamber
\[
        0<t<u<x_3<\cdots<x_s,
\]
we have
\[
        x_j^2-t^2\le x_j^2,
        \qquad
        x_j^2-u^2\le x_j^2,
        \qquad j=3,\ldots,s.
\]
Therefore, after discarding \(e^{-(t^2+u^2)/2}\le1\) and enlarging the domain of the
remaining singular values,
\begin{align}
        \Pbb\{\rho_2\le r\}
        &\le
        \frac{s!}{\mathcal Z_{s,s}}
        \int_{0<t<u<r}(u^2-t^2)\,dt\,du
\nonumber\\
        &\qquad\qquad\times
        \int_{0<x_3<\cdots<x_s}
        e^{-\frac12\sum_{j=3}^s x_j^2}
        \prod_{j=3}^s x_j^4
        \prod_{3\le i<j\le s}(x_j^2-x_i^2)
        \,dx_3\cdots dx_s.
        \label{eq:second-singular-partition-bound}
\end{align}
The second integral is
\[
        \frac{1}{(s-2)!}\mathcal Z_{s-2,s+2}.
\]
Moreover,
\[
        \int_{0<t<u<r}(u^2-t^2)\,dt\,du
        =
        \frac{r^4}{6}.
\]
It follows that
\begin{equation}
        \Pbb\{\rho_2\le r\}
        \le
        \frac{s(s-1)}{6}
        \frac{\mathcal Z_{s-2,s+2}}{\mathcal Z_{s,s}}
        r^4.
        \label{eq:second-singular-partition-ratio}
\end{equation}
Using \eqref{eq:singular-value-partition-evaluation} and canceling common gamma factors gives
\begin{equation}
        \frac{\mathcal Z_{s-2,s+2}}{\mathcal Z_{s,s}}
        =
        \frac12.
        \label{eq:second-singular-partition-ratio-value}
\end{equation}
Substitution into \eqref{eq:second-singular-partition-ratio} proves
\[
        \Pbb\{\rho_2\le r\}
        \le
        \frac{s(s-1)}{12}r^4.
\]
For \(0\le r\le r_0\le1\), this also implies
\[
        \Pbb\{\rho_2\le r\}
        \le
        Cs^2r^3.
\]
This proves part~(c).
\end{proof}

\begin{corollary}[Uniform square hard-edge scaling]
\label{cor:uniform-square-hard-edge-scaling}
There is an absolute constant \(C\) such that, uniformly for
\(0\le t\le r_0\sqrt{s}\),
\begin{equation}
        \Pbb\left\{
        \rho_1\le\frac{t}{\sqrt{s}}
        \right\}
        =
        t
        +
        O\left(
        \frac{t}{s}+t^2
        \right).
        \label{eq:uniform-square-hard-edge-scaled}
\end{equation}
In particular, if \(t_s\to0\), then
\begin{equation}
        \Pbb\left\{
        \rho_1\le\frac{t_s}{\sqrt{s}}
        \right\}
        =
        (1+o(1))t_s.
        \label{eq:triangular-square-hard-edge}
\end{equation}
\end{corollary}

\begin{proof}
Substitute \(r=t/\sqrt{s}\) into
\eqref{eq:uniform-square-hard-edge-expansion}:
\[
        \Pbb\left\{
        \rho_1\le\frac{t}{\sqrt{s}}
        \right\}
        =
        \frac{\gamma_s}{\sqrt{s}}t+O(t^2).
\]
Now use
\[
        \frac{\gamma_s}{\sqrt{s}}
        =
        1+O(s^{-1}).
\]
\end{proof}

\begin{remark}
The estimates above are finite-dimensional and uniform in \(s\); no interchange between
the limits \(r\downarrow0\) and \(s\to\infty\) is required.  The one-value estimate
\eqref{eq:uniform-square-hard-edge-expansion} and the two-value estimate
\eqref{eq:second-singular-value-r3} are the inputs used in
Lemma~\ref{lem:specialized-inverse-tail}.  The corollary also justifies the triangular-array
asymptotic for a single square minor at the exponentially small thresholds used in the
high-probability theorem.
\end{remark}

\section{Rectangular Gaussian hard edge and inverse moments}
\label{app:rectangular-hard-edge}

Let
\[
        C\in\R^{(M-s)\times M},
        \qquad
        1\le s\le M-1,
\]
be standard Gaussian, and let
\[
        \mu_s(C):=\sigma_{M-s}(C)
\]
denote its smallest nonzero singular value.  We first derive the rectangular hard-edge
estimate used in the main text directly from the real Wishart eigenvalue density.  In
particular, no additive exponentially small error term is needed in the Gaussian case.

\begin{lemma}[Rectangular Gaussian hard-edge estimate]
\label{lem:rectangular-gaussian-hard-edge}
There are absolute constants \(C,c>0\) such that, for every
\(1\le s\le M-1\) and \(0<u\le c\),
\begin{equation}
        \Pbb\left\{
        \mu_s(C)\le u\frac{s}{\sqrt M}
        \right\}
        \le
        (Cu)^{s+1}.
        \label{eq:rectangular-gaussian-hard-edge}
\end{equation}
\end{lemma}

\begin{proof}
Set
\[
        n:=M-s.
\]
The nonzero squared singular values of \(C\) are the eigenvalues of the \(n\times n\)
Wishart matrix
\[
        W:=CC^\top.
\]
Write these eigenvalues in increasing order as
\[
        0<\lambda_1<\lambda_2<\cdots<\lambda_n.
\]
Thus
\[
        \mu_s(C)^2=\lambda_1.
\]

Use the convention \(Z_{0,N}=1\) for empty-dimensional integrals.
For later use, define the Laguerre partition function
\[
        Z_{n,M}
        :=
        \int_{(0,\infty)^n}
        e^{-\frac12\sum_{i=1}^n\lambda_i}
        \prod_{i=1}^n\lambda_i^{(s-1)/2}
        \prod_{1\le i<j\le n}|\lambda_j-\lambda_i|
        \,d\lambda_1\cdots d\lambda_n.
\]
The ordered eigenvalues have density
\begin{equation}
        \frac{n!}{Z_{n,M}}
        e^{-\frac12\sum_{i=1}^n\lambda_i}
        \prod_{i=1}^n\lambda_i^{(s-1)/2}
        \prod_{1\le i<j\le n}(\lambda_j-\lambda_i)
        \label{eq:ordered-wishart-density}
\end{equation}
on the chamber
\[
        0<\lambda_1<\cdots<\lambda_n.
\]

Fix \(x>0\) and write \(t=\lambda_1\).  On the ordered chamber,
\[
        0<t<\lambda_j,
        \qquad j=2,\ldots,n,
\]
and therefore
\[
        \lambda_j-t\le\lambda_j.
\]
Using this inequality for the factors coupling \(\lambda_1\) to the remaining eigenvalues,
discarding \(e^{-t/2}\le1\), and enlarging the domain of the remaining eigenvalues from
\(t<\lambda_2<\cdots<\lambda_n\) to
\(0<\lambda_2<\cdots<\lambda_n\), we obtain
\begin{align}
        \Pbb\{\lambda_1\le x\}
        &\le
        \frac{n!}{Z_{n,M}}
        \int_0^x t^{(s-1)/2}\,dt
\nonumber\\
        &\qquad\qquad\times
        \int_{0<\lambda_2<\cdots<\lambda_n}
        e^{-\frac12\sum_{j=2}^n\lambda_j}
        \prod_{j=2}^n\lambda_j^{(s+1)/2}
        \prod_{2\le i<j\le n}(\lambda_j-\lambda_i)
        \,d\lambda_2\cdots d\lambda_n.
        \label{eq:wishart-smallest-bound-before-partition}
\end{align}
The second integral is
\[
        \frac{1}{(n-1)!}Z_{n-1,M+1}.
\]
Indeed, an \((n-1)\times(n-1)\) Wishart matrix with \(M+1\) degrees of freedom has
hard-edge exponent
\[
        \frac{(M+1)-(n-1)-1}{2}
        =
        \frac{s+1}{2}.
\]
Since
\[
        \int_0^x t^{(s-1)/2}\,dt
        =
        \frac{2}{s+1}x^{(s+1)/2},
\]
equation \eqref{eq:wishart-smallest-bound-before-partition} gives
\begin{equation}
        \Pbb\{\lambda_1\le x\}
        \le
        B_{M,s}\,x^{(s+1)/2},
        \qquad
        B_{M,s}
        :=
        \frac{2n}{s+1}
        \frac{Z_{n-1,M+1}}{Z_{n,M}}.
        \label{eq:wishart-smallest-partition-ratio}
\end{equation}

The Laguerre--Selberg integral gives
\begin{equation}
        Z_{n,M}
        =
        2^{nM/2}
        \prod_{j=1}^n
        \frac{
        \Gamma(1+j/2)\Gamma((s+j)/2)
        }{
        \Gamma(3/2)
        }.
        \label{eq:laguerre-partition-function}
\end{equation}
Applying the same formula with dimension \(n-1\) and \(M+1\) degrees of freedom,
and canceling the common factors, yields
\begin{equation}
        \frac{Z_{n-1,M+1}}{Z_{n,M}}
        =
        2^{-(s+1)/2}
        \frac{
        \Gamma(3/2)\Gamma((M+1)/2)
        }{
        \Gamma(1+n/2)
        \Gamma((s+1)/2)
        \Gamma((s+2)/2)
        }.
        \label{eq:laguerre-partition-ratio}
\end{equation}
By the duplication formula,
\[
        \Gamma\left(\frac{s+1}{2}\right)
        \Gamma\left(\frac{s+2}{2}\right)
        =
        2^{-s}\sqrt{\pi}\,\Gamma(s+1),
\]
and \(\Gamma(3/2)=\sqrt{\pi}/2\).  Hence
\begin{equation}
        B_{M,s}
        =
        \frac{n\,2^{(s-1)/2}}{s+1}
        \frac{
        \Gamma((M+1)/2)
        }{
        \Gamma((n+2)/2)\Gamma(s+1)
        }.
        \label{eq:BMs-explicit}
\end{equation}

An elementary consequence of Stirling's bounds is
\[
        \frac{\Gamma((M+1)/2)}
        {\Gamma((n+2)/2)}
        \le
        (CM)^{(s-1)/2},
\]
uniformly for \(n=M-s\ge1\).  Using \(n\le M\) and
\[
        \Gamma(s+1)=s!\ge\left(\frac{s}{e}\right)^s,
\]
equation \eqref{eq:BMs-explicit} implies
\begin{equation}
        B_{M,s}
        \le
        \left(\frac{C\sqrt M}{s}\right)^{s+1}.
        \label{eq:BMs-bound}
\end{equation}
Here and below the absolute constant \(C\) may increase from line to line.

Substituting \(x=\varepsilon^2\) into
\eqref{eq:wishart-smallest-partition-ratio} and using
\eqref{eq:BMs-bound}, we conclude that
\begin{equation}
        \Pbb\{\mu_s(C)\le\varepsilon\}
        \le
        \left(
        \frac{C\sqrt M}{s}\varepsilon
        \right)^{s+1}.
        \label{eq:rectangular-gaussian-hard-edge-unscaled}
\end{equation}
Taking
\[
        \varepsilon=u\frac{s}{\sqrt M}
\]
proves \eqref{eq:rectangular-gaussian-hard-edge}.
\end{proof}

\begin{remark}
The scale in \eqref{eq:rectangular-gaussian-hard-edge} is equivalent, up to absolute
constants, to the usual rectangular least-singular-value scale
\[
        \sqrt M-\sqrt{M-s}
        =
        \frac{s}{\sqrt M+\sqrt{M-s}}.
\]
The exponent \(s+1\) is the real hard-edge codimension exponent.  General subgaussian
smallest-singular-value theorems give the same scale and exponent, often with an additional
term exponentially small in \(M\).  The direct Gaussian calculation above removes that
additional term, which is useful here because the threshold \(u\) is itself exponentially
small.
\end{remark}

\begin{corollary}[Central inverse-operator moments]
\label{cor:central-inverse-operator-moments}
Fix \(\eta\in(0,1/2)\).  Let
\[
        C\in\R^{m\times M}
\]
be standard Gaussian, and set
\[
        s:=M-m.
\]
Assume that
\[
        \frac{s}{M}\in[\eta,1-\eta].
\]
Then, for every fixed \(0<p<\infty\), there are constants
\(M_0=M_0(\eta,p)\) and \(C_{\eta,p}<\infty\) such that, for
\(M\ge M_0\),
\begin{equation}
        \Ebb
        \left\|(CC^\top)^{-1/2}\right\|^p
        =
        \Ebb\,\sigma_m(C)^{-p}
        \le
        C_{\eta,p}M^{-p/2}.
        \label{eq:central-inverse-operator-moments}
\end{equation}
More generally, the estimate holds whenever \(s+1>p\), with a constant that may
depend on \(p\) and on a lower bound for \(s/M\).
\end{corollary}

\begin{proof}
Let
\[
        \mu:=\sigma_m(C),
        \qquad
        Y:=\frac{\sqrt M}{\mu}.
\]
It is enough to show that
\[
        \Ebb Y^p\le C_{\eta,p}.
\]

By Lemma~\ref{lem:rectangular-gaussian-hard-edge},
\[
        \Pbb\left\{
        \mu\le u\frac{s}{\sqrt M}
        \right\}
        \le
        (Cu)^{s+1},
        \qquad
        0<u\le c.
\]
For \(v>0\), the event \(\{Y\ge v\}\) is
\[
        \left\{
        \mu\le\frac{\sqrt M}{v}
        \right\}.
\]
Writing
\[
        \frac{\sqrt M}{v}
        =
        \left(\frac{M}{sv}\right)\frac{s}{\sqrt M},
\]
we obtain
\[
        \Pbb\{Y\ge v\}
        \le
        \left(\frac{CM}{sv}\right)^{s+1}
        \le
        \left(\frac{C_\eta}{v}\right)^{s+1}
\]
whenever
\[
        v\ge v_\eta:=\frac{1}{c\eta}.
\]
After increasing \(C_\eta\), we may therefore write
\begin{equation}
        \Pbb\{Y\ge v\}
        \le
        \min\left\{
        1,
        \left(\frac{C_\eta}{v}\right)^{s+1}
        \right\},
        \qquad v>0.
        \label{eq:central-inverse-tail}
\end{equation}

Choose \(L_\eta\ge\max\{v_\eta,2C_\eta\}\).  By the tail-integral formula,
\begin{align*}
        \Ebb Y^p
        &=
        p\int_0^\infty
        v^{p-1}\Pbb\{Y\ge v\}\,dv
\\
        &\le
        L_\eta^p
        +
        pC_\eta^{s+1}
        \int_{L_\eta}^\infty
        v^{p-s-2}\,dv.
\end{align*}
If \(s+1>p\), then
\[
        \int_{L_\eta}^\infty
        v^{p-s-2}\,dv
        =
        \frac{L_\eta^{p-s-1}}{s+1-p}.
\]
Consequently,
\[
        \Ebb Y^p
        \le
        L_\eta^p
        +
        \frac{pL_\eta^p}{s+1-p}
        \left(\frac{C_\eta}{L_\eta}\right)^{s+1}
        \le
        C_{\eta,p}.
\]
For fixed \(p\), the condition \(s+1>p\) holds for all sufficiently large \(M\)
because \(s\ge\eta M\).  Finally,
\[
        \mu^{-p}
        =
        M^{-p/2}Y^p,
\]
which proves \eqref{eq:central-inverse-operator-moments}.
\end{proof}

\begin{corollary}[Moments of the row-conditioned factor]
\label{cor:row-conditioned-factor-moments}
Under the assumptions of
Corollary~\ref{cor:central-inverse-operator-moments}, let
\(S\in\R^{m\times m}\) have the same singular values as \(C\), and let
\[
        G_1\in\R^{s\times m}
\]
be an independent standard Gaussian matrix.  Then, for every fixed
\(0<p<\infty\) and all sufficiently large \(M\),
\begin{equation}
        \Ebb\|S^{-1}\|^p
        \le
        C_{\eta,p}M^{-p/2},
        \label{eq:central-S-inverse-moments}
\end{equation}
and
\begin{equation}
        \Ebb\|G_1S^{-1}\|^p
        \le
        C_{\eta,p}.
        \label{eq:central-conditioned-factor-moments}
\end{equation}
Consequently, for
\[
        K:=(-G_1S^{-1},I_s),
\]
one has
\begin{equation}
        \Ebb\|K\|^p\le C_{\eta,p}.
        \label{eq:central-K-moments}
\end{equation}
\end{corollary}

\begin{proof}
The first estimate follows directly from
Corollary~\ref{cor:central-inverse-operator-moments}.  For the second, Hölder's
inequality gives
\[
        \Ebb\|G_1S^{-1}\|^p
        \le
        \left(\Ebb\|G_1\|^{2p}\right)^{1/2}
        \left(\Ebb\|S^{-1}\|^{2p}\right)^{1/2}.
\]
The standard Gaussian operator-norm tail
\[
        \Pbb\{\|G_1\|\ge \sqrt{s}+\sqrt{m}+t\}
        \le e^{-t^2/2},
        \qquad t\ge0,
\]
implies, by tail integration,
\[
        \Ebb\|G_1\|^{2p}
        \le C_p(\sqrt{s}+\sqrt{m})^{2p}
        \le C_{\eta,p}M^p.
\]
On the other hand,
\[
        \Ebb\|S^{-1}\|^{2p}
        \le
        C_{\eta,p}M^{-p}.
\]
Thus
\[
        \Ebb\|G_1S^{-1}\|^p\le C_{\eta,p}.
\]
Finally,
\[
        \|K\|
        \le
        1+\|G_1S^{-1}\|,
\]
which proves \eqref{eq:central-K-moments}.
\end{proof}

We now record the inverse-moment consequences used in the overlap bound.

\begin{lemma}[Truncated inverse moments]
\label{lem:rectangular-truncated-inverse-moments}
Let
\[
        \mu:=\mu_s(C).
\]
There are absolute constants \(C,c>0\) with the following properties.

If \(s\ge2\) and
\[
        0<a\le c\frac{s}{\sqrt M},
\]
then
\begin{equation}
        \Ebb
        \min\left\{
        1,\frac{a^2}{\mu^2}
        \right\}
        \le
        C a^2\frac{M}{s^2}.
        \label{eq:rectangular-truncated-inverse-moment}
\end{equation}

If \(s=1\) and
\[
        0<a\le\frac{c}{\sqrt M},
\]
then
\begin{equation}
        \Ebb
        \min\left\{
        1,\frac{a^2}{\mu^2}
        \right\}
        \le
        C a^2M
        \log\left(\frac{e}{a\sqrt M}\right).
        \label{eq:rectangular-endpoint-inverse-moment}
\end{equation}
\end{lemma}

\begin{proof}
Set
\[
        X:=\frac{\sqrt M}{s}\mu,
        \qquad
        b:=\frac{a\sqrt M}{s}.
\]
After decreasing the absolute constant \(c>0\) if necessary,
Lemma~\ref{lem:rectangular-gaussian-hard-edge} gives
\begin{equation}
        \Pbb\{X\le u\}
        \le
        (Cu)^{s+1},
        \qquad
        0<u\le c,
        \label{eq:normalized-rectangular-hard-edge}
\end{equation}
where \(Cc\le1/2\).

For any positive random variable \(X\), the layer-cake formula gives
\begin{equation}
        \Ebb\min\left\{1,\frac{b^2}{X^2}\right\}
        =
        2b^2
        \int_b^\infty
        u^{-3}\Pbb\{X\le u\}\,du.
        \label{eq:truncated-inverse-layer-cake}
\end{equation}
Split the integral at \(c\).  For \(s\ge2\), use
\eqref{eq:normalized-rectangular-hard-edge} on \([b,c]\) and the trivial bound by \(1\)
on \([c,\infty)\):
\begin{align*}
        \Ebb\min\left\{1,\frac{b^2}{X^2}\right\}
        &\le
        2b^2C^{s+1}
        \int_b^c u^{s-2}\,du
        +
        2b^2\int_c^\infty u^{-3}\,du
\\
        &\le
        Cb^2.
\end{align*}
The constant is uniform in \(s\ge2\) because \(Cc\le1/2\).  Since
\[
        b^2=a^2\frac{M}{s^2},
\]
this proves \eqref{eq:rectangular-truncated-inverse-moment}.

When \(s=1\), the integral over \([b,c]\) becomes logarithmic:
\[
        2b^2C^2
        \int_b^c\frac{du}{u}
        \le
        Cb^2\log\left(\frac{e}{b}\right).
\]
Since \(b=a\sqrt M\) in this case, this proves
\eqref{eq:rectangular-endpoint-inverse-moment}.
\end{proof}

\begin{remark}
In the application to the overlap profile, the singular inverse-moment term is
\[
        \Ebb
        \min\left\{
        1,\frac{t^2s}{\mu_s(C)^2}
        \right\}.
\]
Thus one applies Lemma~\ref{lem:rectangular-truncated-inverse-moments} with
\[
        a=t\sqrt{s}.
\]
The required range is therefore
\[
        t\le c\sqrt{\frac{s}{M}}.
\]
\end{remark}

\section{Inverse-Wishart trace estimates}
\label{app:inverse-wishart}

Let \(C\in\R^{m\times M}\) be standard Gaussian, \(W=CC^\top\), and put
\(d=M-m\).  In the main text, \(m=M-s\) and \(d=s\).  The matrix \(W\) has a real
Wishart distribution with dimension \(m\), degrees of freedom \(M\), and identity scale.
The required trace estimates are standard inverse-Wishart consequences.  General inverse
Wishart moments are computed explicitly by Matsumoto \cite{Matsumoto2012}; the invariant
moment framework of Letac and Massam \cite{LetacMassam2004} gives another standard
source.

For identity scale and \(d>3\), the inverse-Wishart formulas give
\[
        \Ebb W^{-1}=\frac{I_m}{d-1},
        \qquad
        \Ebb\tr(W^{-1})=\frac{m}{d-1}.
\]
They also give, for the entries of \(X=W^{-1}\),
\[
        \Var(X_{ii})=\frac{2}{(d-1)^2(d-3)},
        \qquad
        \Cov(X_{ii},X_{jj})=\frac{2}{d(d-1)^2(d-3)}\quad(i\ne j),
\]
and
\[
        \Ebb X_{ij}^2=\frac{1}{d(d-1)(d-3)}\quad(i\ne j).
\]
Therefore
\[
        \Var\bigl(\tr(W^{-1})\bigr)
        =
        \frac{2m}{(d-1)^2(d-3)}
        +
        \frac{2m(m-1)}{d(d-1)^2(d-3)}.
\]
When \(d/M\in[\eta,1-\eta]\), this variance is \(O_\eta(M^{-2})\).  Hence
\[
        \tr(W^{-1})=\frac{m}{d}+o_{L^2}(1).
\]
Similarly,
\begin{align*}
        \Ebb\tr(W^{-2})
        &=\sum_i \Ebb X_{ii}^2+
          \sum_{i\ne j}\Ebb X_{ij}^2 \\
        &=\frac{m(M-1)}{d(d-1)(d-3)}
          \le C_\eta M^{-1},
\end{align*}
again uniformly when \(d/M\in[\eta,1-\eta]\).  These are precisely
\eqref{eq:tr-W-inv-conc} and \eqref{eq:tr-W-inv2} in the main text.

\section{Hypergeometric overlap tails}
\label{app:hypergeometric}

Fix an \(M\)-subset \(T\subset[2M-1]\), and let \(U\) be uniformly distributed among all
\(M\)-subsets.  The overlap-difference parameter
\[
        S=|T\setminus U|
\]
has mass
\[
        \Pbb\{S=s\}=\frac{\binom Ms\binom{M-1}s}{\binom{2M-1}{M}}.
\]
Equivalently, \(M-S=|T\cap U|\) is hypergeometric with population size \(2M-1\),
number of marked elements \(M\), and sample size \(M\).  Its mean is
\[
        \Ebb S=M-\frac{M^2}{2M-1}=\frac{M(M-1)}{2M-1}=\frac M2+O(1).
\]
Standard hypergeometric Chernoff bounds imply that, for every fixed \(\eta\in(0,1/2)\),
\[
        \Pbb\{S\notin[\eta M,(1-\eta)M]\}\le e^{-c_\eta M}
\]
for all sufficiently large \(M\).  Consequently,
\[
        \frac1{K_M}\sum_{\substack{1\le s\le M-1\\s\notin[\eta M,(1-\eta)M]}}
        \binom Ms\binom{M-1}s\frac Ms
        \longrightarrow0,
\]
and the endpoint \(s=1\) contribution is exponentially negligible in all places where it
is used in the main text.

\section*{Acknowledgments}
The author thanks Afonso Bandeira and the Randomstrasse101 blog for bringing
the critical-threshold stability problem in phase retrieval to the author's
attention.

\newcommand{\etalchar}[1]{$^{#1}$}

\end{document}